\documentclass[12pt,a4paper,twoside]{article}

\usepackage[latin1]{inputenc}
\usepackage{amsmath}
\usepackage{amsfonts}
\usepackage{amsthm}
\usepackage{amssymb}
\usepackage{graphicx}
\usepackage{color}
\newcommand{\de}{\partial}
\newcommand{\dl}{\delta}
\newcommand{\De}{\Delta}
\newcommand{\Om}{\Omega}
\newcommand{\eL}{\mathcal{L}}
\newcommand{\eF}{\mathcal{F}}
\newcommand{\intt}{\int_0^T}
\newcommand{\intx}{\int_\Om}
\newcommand{\half}[1]{\frac{#1}{2}}
\newcommand{\sigs}{\sigma^2}
\newcommand{\Th}{\Theta}
\newcommand{\Nth}{N_{\Theta}}
\newcommand{\R}{\mathbb{R}}
\def\hJ {{\hat J}}

\def\dt{{\delta t}}

\newtheorem{algorithm}{Algorithm}

\newtheorem{rem}{Remark}
\newtheorem{thm}{Theorem}
\newtheorem{lem}{Lemma}

\title{Calibration of L\'evy Processes using Optimal Control of Kolmogorov Equations with Periodic Boundary Conditions}

\author{M.~Annunziato\thanks{Dipartimento di Matematica, Universit\`a degli 
Studi di Salerno, 
Via Giovanni Paolo II, 132 - 84084 Fisciano, Italia (\tt{mannunzi@unisa.it})}
\and
H.~Gottschalk\thanks{FBC - Fachgruppe f\"ur Mathematik und Informatik, Gau\ss str.\ 20, 42119 Wuppertal,
Germany ({\tt hanno.gottschalk@uni-wuppertal.de})}}

\begin{document}

\markboth{M. Annunziato and H. Gottschalk}{Calibration of L\'evy Processes}


\maketitle

\begin{abstract} 
We present an optimal control approach to the problem of model calibration for L\'evy processes based on a non parametric estimation procedure.
The calibration problem is of considerable interest in mathematical finance and beyond. 
Calibration of L\'evy processes is particularly challenging as the jump distribution is given by an arbitrary L\'evy measure, which form a infinite dimensional space. 
In this work, we follow an approach which is related to the maximum likelihood theory of sieves \cite{GH}.  
 The sampling of the L\'evy process is modelled as independent
observations of the stochastic process at some terminal time $T$.  
We use a generic spline discretization of the L\'evy jump measure and select an adequate size of the spline basis using the Akaike Information Criterion (AIC) \cite{BA}. 

The numerical solution of the L\'evy calibration problem requires efficient optimization of the log likelihood functional in high dimensional parameter spaces. 
We provide this by the optimal control of Kolmogorov's forward equation for the probability density function (Fokker-Planck equation). 
The first order optimality conditions are derived based on the Lagrange multiplier technique in a functional space. 
The resulting Partial Integral-Differential Equations (PIDE) are discretized, numerically solved and controlled using scheme a composed of Chang-Cooper, BDF2 and
direct quadrature methods. 
For the numerical solver of the Kolmogorov's forward equation we prove conditions for 
non-negativity and stability in the $L^1$ norm of the discrete solution.
To set boundary conditions, we argue that any L\'evy process on the real line can be projected to a torus, where it again is a L\'evy process. If the torus is sufficiently large, the loss of information is negligible.        
\end{abstract}

\noindent \textbf{MSC (2010):} 93E10 (primary) 49K20, 60G51, 62G05 (secondary)

\vspace{5mm}

\noindent \textbf{Key words:} Optimal control of PIDE, Kolmogorov equations, Fokker-Planck equation, L\'evy processes, non-parametric maximum likelihood method, Akaike information criterion, financial data.

\section{Introduction}
L\'evy processes play a large r\^ole in contemporary mathematical finance \cite{CT}, but also in many areas of physics, see e.g. \cite{App,Sak}. A real valued L\'evy process is a stochastic process $Y(t)$ that has increments $Y(t)-Y(s)$, $t\geq s$, that are independent of the past. The increments are also stationary in the sense that the probability distribution of the increment only depends on the time difference $t-s$. Furthermore, $Y(0)=0$ and a stochastic continuity condition for $t=0$ holds, see e.g. \cite{App}. Under the given conditions, the characteristic function of $Y(t)$ is given by the L\'evy Khinchine representation
\begin{equation} \label{eqa:LevyKhinchine}
\mathbb{E}\left[e^ {iY(t) k}\right]=e^{t\psi(k)}
\end{equation}
$\mathbb{E}[\cdot]$ stands for the expected value.
$\psi(k)$ is a conditionally positive definite function \cite{BF} that has the following representation in terms of the canonical triplet $(b,\sigma^2,\nu)$:
\begin{equation}
\label{eqa:LevyFormula}
\psi(k)=ibk-\frac{\sigma^2}{2}k^2+\int_{\mathbb{R}\setminus \{0\}}\left(e^{isk}-1-isk 1_{\{|s|\leq 1\}}(s)\right)d\nu(s).
\end{equation} 
$b,\sigma^2\in \R$ are constants, $\sigma^2\geq 0$, and the L\'evy measure $\nu$ is a positive measure on $\R\setminus\{0\}$ such that 
\begin{equation}
\label{eqa:LevyCondition}
\int_{\mathbb{R}\setminus \{0\}} \min(1,s^2)d\nu(s)<\infty.
\end{equation}
In Eq. (\ref{eqa:LevyFormula}) $1_{\{|s|\leq 1\}}(s)$ is the characteristic function of the set $\{|s|\leq 1\}$ which takes the value $1$ on this set and $0$ otherwise.

The calibration problem for L\'evy processes consists of the estimation of the canonical triplet $(b,\sigma^2,\nu)$ given the observation $Y(t_j)$ of the process' trajectory $Y(t)$ at some prescribed times $t_j$, $j=1,\ldots,L$. For instance, $Y(t)$ could be the process of log-Returns of some asset and $t_j$ could be the closing time of the $j$-th trading day  ('historic low frequency data'). As the $j$-th increment of the process $Y_j=Y(t_j)-Y(t_{j-1})$ has the same distribution as $Y(T)$, if $T=t_{j}-t_{j-1}$, from a statistical point of view this is equivalent to the $L$-fold independent observation of the terminal values $Y(T)$ at time $T$.

$Y(t)$ can also be understood as the solution to the Stochastic Differential Equation (SDE) of jump-diffusion type
\begin{equation}
\label{eqa:SDE}
dY(t)=bdt+\sigma dB(t) +\int_{\{|y|\leq 1\}}y \tilde N(dt,dy)+\int_{\mathbb{R}\setminus\{|y|> 1\}} y N(dt,dy),~~Y(0)=0.
\end{equation}
Here $B(t)$ is a standard Brownian motion and $ N((t,t+\Delta t],A)\sim{\rm Po}(\nu(A)\Delta t)$ is the random counting measure of jumps of height in the set $A\subseteq \mathbb{R}$ in the time interval $(t,t+\Delta t]$. ${\rm Po}(\lambda)$ stands for the Poisson distribution with intensity $\lambda$ and $\tilde N((t,t+\Delta t],A)=N((t,t+\Delta t],A)-\nu(A)\Delta t$ is the compensated or Martingale jump measure for small jumps, where we require $A\subset \mathbb{R}\setminus \{|x|\leq\varepsilon\}$ for some $\varepsilon>0$, see \cite{App} for further details.
   
The calibration problem for L\'evy processes, respectively the solution of (\ref{eqa:SDE}),  unfortunately is ill posed: The collection of all L\'evy measures $ \nu$ is infinite dimensional, while only $L$ observations are available. Direct application of the maximum likelihood principle in this situation leads to severe over-fitting issues \cite{GH}. In many applications, one  chooses families of L\'evy measures $\nu(\bar\alpha)$ that depend only on a finite dimensional parameter vector $\bar\alpha$, see e.g. \cite{Iac}.  Furthermore, one often restricts to such parametrizations, where the density $f(x,T,\bar\alpha)$ of the probability distribution of $Y(t)$ can be calculated explicitly or at low numerical cost. One then assumes that the true distribution of $Y(t)$ is inside the prescribed set and uses the maximum likelihood approach for calibration \cite{Fer}.  This assumption might however not be justified and give rise to modelling errors.

As a non-parametric alternative, one can use generic parametrizations for the density of the L\'evy measure $\nu$ that can be  refined depending on the amount of data available. This gives rise to a hierarchy -- or sieve \cite{GH} -- of maximum likelihood problems with a finite number of parameters. If a suitable finite parametrization has been chosen, it remains to solve the maximum likelihood estimate at a given level of parametrization. One also has to determine this level based on the quality and also stability of the fits obtained. The resulting  densities can no longer be calculated analytically. Also, solution of the maximum likelihood problem gives rise to high dimensional optimization problems. 

The maximum likelihood method requires a parametric representation of the probability density functions (PDF). The PDF can however be obtained as a solution to the Kolmogorov forward equation (Fokker Planck equation). The parameters $\bar\alpha$ then enter in this equation via coefficients in the generator of the semigroup \cite{App}.  
If the L\'evy measure $\nu$ is not zero,  the generator of both these equations does not only contain a  2nd order partial differential operator, but also an integral operator of convolution type.  This places the model calibration problem in the framework of 
optimal control problems with partial integral differential equations (PIDE) constraints. 

Indeed, we know that the Kolmogorov forward equation is representative of a stochastic process described in terms of SDEs such as that one of Eq. (\ref{eqa:SDE}), where the set of parametrization for the approximation of the PDF, would correspond to a set of controls of the stochastic dynamic equation, so that, jointly to the maximum likelihood problem, it corresponds to a stochastic optimal control problem.
The classical way to deal with the optimal control of stochastic process is by the Dynamic Programming principle and the related Hamilton-Jacobi-Bellman equation for stochastic processes \cite{ber}. 
However, this problem  has been recently framed as a constrained
PDE optimization problem, where the PDE is the Fokker-Planck, i.e. Kolmogorov forward, equation
\cite{mar:alf:mma,mar:alf,hjb_fp}. Following this framework, the solution of the maximum 
likelihood problem, i.e. the stochastic optimization, is found by solving the first order
optimality conditions in a functional space, that is the optimality system 
consisting of two PIDEs, named forward and backward (or adjoint) equations, plus an optimality condition.

This optimality system can be numerically solved by a gradient-based iterative algorithm as follows.
The Kolmogorov forward equation has a set of control parameters in order to maximize the log-Likelihood functional for its terminal value. These controls involve the Kolmogorov backward equation  (adjoint equation) with suitable terminal condition, corresponding to the log-likelihood functional. 
Hence, given an initial approximation of the unknown parametrization, first solve the forward equation,
then  set up the terminal condition and solve the adjoint one. With both the forward and adjoint solutions, by using the optimality condition equation the gradient is computed. Then with a descending gradient technique, such as a non linear conjugate gradient method, found a better approximation of the 
control parameters and repeat until the satisfying accuracy for the parametrization is found.

Since this maximum likelihood problem could have an high dimensional space 
ad a huge number $L$ of observations, a fast, stable and enough accurate numerical solver for the
PIDE is required.
In our case the Kolmogorov forward equation is a PIDE of parabolic differential operator type. 
Such kind of PIDEs are, e.g., investigated in the option pricing models
as a generalization of the Black-Scholes equation. The first difficulty to numerically 
solve this equation is the integral. In fact, in the case of using a fully implicit
method, it would lead to solve a dense system of equation, for this reason implicit-explicit (IMEX) or operator splitting methods can be applied to bypass this problem (see \cite{duf,bri,con:vol}).
The solution of the Kolmogorov forward equation is a probability density function
that is non negative with constant integral over the domain. Such properties must be owned from the discrete solution too. The Chang-Cooper (CC) is a non-negative and conservative numerical method 
that has been used to solve the classical Fokker-Planck equation \cite{chacoo,mar:alf,bor:moh}.
Here, we use a numerical method that can be classified as IMEX, since we use the CC method with an
implicit time difference scheme for the differential operators of our PIDE, and
evaluate the integral operator at the previous time step solution, i.e. in an explicit way. We prove for the resulting numerical solver: conservativeness, non-negativity and stability in the $L^1$-norm.
The numerical solver for the adjoint equation is obtained directly from the solver for the 
forward equation by using the ``discretize then optimize'' approach to the optimization problem.

Finally, we quote, that for related work with vanishing L\'evy measure see e.g. \cite{mar:alf,BS}, and for estimation procedures based on non parametric approximations of the empirical characteristic function see e.g. \cite{BR}. An approach based on the method of moments and asymptotic expansions of L\'evy densities can be found in \cite{GST}.

The article is organised as follows: In the following Section \ref{sec:Reg} we describe the hierarchy, of estimation problems. We shall show that the estimation problems that can actually be solved numerically can come arbitrarily close, at increasing computational cost, to the fully general L\'evy estimation problem. We also show that the use of periodic boundary conditions in the Kolmogorov equations can be understood in terms of mapping the original L\'evy process on the real line to a derived L\'evy process on the torus. 

In Section \ref{sec:Kol}, we set up the maximum likelihood estimation problem for a given parametrization and derive Kolmogorov's forward (Fokker-Planck) equation and its adjoint (Kolmogorov backward) equation with terminal conditions set by the log-Likelihood objective functional. 
This maximum likelihood estimation problem is solved in the framework of the 
Fokker-Planck optimal control of stochastic processes, as a constrained PDE optimal control problem.

In Section \ref{sec:Num} the discretization for Kolmogorov's equations and the optimal control scheme is derived following a Chang-Cooper and IMEX approach.   
In particular we prove the structural properties of the numerical solution, i.e. conservativeness,
non-negativity and stability.

Section \ref{sec:Cons} gives numerical tests for the consistency of the proposed procedure based on  simulated data. We propose to use Akaike's information criterion (AIC) \cite{BA} to choose an adequate parametrization from the hierarchy of spline parametrizations for density of the L\'evy measure. Three different tests are performed:   We first fit data that are simulated from a given distribution within our hierarchy of L\'evy distributions. The fits obtained are shown to be of very good quality and AIC-selection criterion reproduces almost the original parametrization. As a second test we fit simulated data from a bi-directional Gamma process, i.e. the difference of two independent Gamma processes \cite{App}, which is not inside one of the parametrizations of degree $N_\Theta$, but can be approximated by those. The bi-directional Gamma process is augmented by a small diffusive component and projected to the torus. The AIC selection criterion and the fitting results again reproduce the final distribution of this process rather adequately. As a final test, we select financial data from the German stock exchange DAX in a period between  April 1998 and March 2002 and consider daily log-returns over 1000 trading days. This period represents a rather stable period for the DAX with an almost constant level of the volatility.  After projection, the AIC based method selects a six-parameter spline approximation of the L\'evy measure density. The resulting fits again give a decent reproduction of the empirical distribution.     

Our conclusions and an outlook are given in the final Section \ref{sec:Out}.  

\section{\label{sec:Reg}A Hierarchy of Parametrizations for Maximum Likelihood Estimation of L\'evy Data}

The estimation problem for the L\'evy measure $\nu$ is plagued by several issues. Here we take a step by step approach towards the derivation of a hierarchy of estimation problems that approximate the original one. 

Let $f(x|\bar \alpha)=f(x|\bar \alpha,N_\Theta)$, $\bar\alpha\in \mathbb{R}^{N_\Theta}$, be a family of probability density functions with variable dimension $N_\Theta$  of the parameter set 
$\bar\alpha$, and let $f(x)$ be the (unknown) probability density of $Y_t$. 

For $N_\Theta$ fixed, we apply the Maximum Likelihood method to select an estimated value $\hat \alpha$ based on increasing sample size. It is known from the general theory of Maximum Likelihood \cite{Fer}  that the estimated $\hat\alpha$ converges almost surely to the true value $\bar\alpha$, provided   $f(x)=f(x|\bar \alpha,N_\Theta)$ holds. 

This leaves the question open, which parametrization -- or which value for $N_\Theta$-- one should choose. We solve this problem by maximizing the Akaike Information Criterion (AIC). Maximizing the AIC corresponds to minimizing (asymptotically) the expected Kulback-Leibler distance, or relative entropy, between the true distribution $f(x)$ and its parametric estimate $f(x|\bar \alpha,N_\Theta)$. See \cite{BA}[Chapter 6] for a detailed derivation.

Let us now define the parametrizations with $N_\Theta$ parameters. We intend to show that, for sufficiently large $N_\theta$, we can approximate the original L\'evy distribution to an arbitrary precision. This requires several steps of approximation:

\paragraph{Truncating small and large jumps.} The total mass of the L\'evy measure, $\lambda=\nu(\mathbb{R}\setminus\{0\})$, can be infinite.  This quantity defines the average number of jumps per unit time of the associated L\'evy process \cite{App}. The easiest way to deal with this is to truncate small jumps by setting $\nu_\varepsilon=1_{\{|s|>\varepsilon\}}\nu$  which is a finite measure by equation (\ref{eqa:LevyCondition}). Using (\ref{eqa:LevyCondition}) and dominated convergence, one can moreover prove that $\psi_\varepsilon(k)\to \psi(k)$ for $k\in\mathbb{R}$ as $\varepsilon\searrow 0$. Here $\psi_\varepsilon(k)$ is given by (\ref{eqa:LevyFormula}) with $\nu$ replaced by $\nu_\varepsilon$.   By the continuity theorem of Paul L\'evy, see e.g. \cite[Theorem 23.8]{Bauer}, this then implies (weak) convergence in law of the respective probability distributions. 

At the same time, a finite L\'evy measure $\nu$ permits one to re-parametrize $\psi(k)$ of
Eq. (\ref{eqa:LevyFormula}) via
\begin{equation}
\label{eqa:LevyFormulaII}
\psi(k)=ib k-\frac{\sigma^2}{2}k^2+\int_{\mathbb{R}\setminus \{0\}}\left(e^{isk}-1\right)d\nu(s)
\end{equation}
with $b\to b-\int_{\mathbb{R}\setminus \{0\}}s1_{\{|s|\leq 1\}}(s)d\nu(s)$. In the following we assume $\nu$ to be finite and use parametrization (\ref{eqa:LevyFormulaII}). The last term in  (\ref{eqa:LevyFormulaII}) now has the structure of a compound Poisson distribution, i.e. $Y(t)$ can be represented as $Z(t)+\sum_{j=1}^{N(t)}U_j$ where $Z\sim \mathcal{N}(bt,\sigma^2t)$ is normally distributed, $N(t)\sim {\rm Po}(\lambda t)$ is Poisson distributed with intensity $\lambda t$ and $U_j$ are i.i.d. random variables with distribution given by the normalized L\'evy measure, $U_j\sim \lambda^{-1}\nu$. Also $Z(t),N(t)$ and $U_j$ are all stochastically independent.

With a similar argument, we can cut off large jumps by replacing $\nu$ with $1_{\{|s|\leq \varepsilon^{-1}\}}\nu$. Also in this case, in the limit $\varepsilon \searrow 0$, the truncated L\'evy distributions converge in law to the non truncated one. In the following we thus assume that the support of $\nu$ is contained in some finite interval $[\Omega_a,\Omega_b]$.  The appropriate size of this region can be estimated e.g. by the Chebyshev's inequality using empirical mean and variance from the data. 
 
\paragraph{Regularizing the L\'evy measure.} Given a non negative, continuously differentiable function $\chi$ with compact support such that $\int_\mathbb{R}\chi(x) dx=1$, and setting $\chi_\varepsilon(x)=\frac{1}{\varepsilon}\chi(\frac{x}{\varepsilon})$, we define $\varrho_\varepsilon(s)=\chi_\varepsilon*\nu(s)=\int_{\R\setminus\{0\}}\chi_\varepsilon(s-\xi)d\nu(\xi)$. We consider the regularised measures $d\nu_\varepsilon(s)=\varrho_\varepsilon(s)ds$.  Inserting this measure in (\ref{eqa:LevyFormulaII}), using Fubini's theorem and dominated convergence, one easily shows that $\psi_\varepsilon(k)\to\psi(k)$. This again implies convergence of the related probability distributions in law.

\paragraph{Spline approximation of the densities.} Let thus $d\nu(s)=\varrho(s)ds$ with $\rho(s)$ non negative, continuously differentiable and with compact support. Let $\theta_j$, $j=1,\ldots,n+1$, be a collection of uniform grid points in $\mathbb{R}$ such that the support of $\varrho(x)$ is covered and $\theta_j-\theta_{j-1}=\Delta$. Define $\varrho_\Delta(s)$ by the linear interpolation between points $(\theta_j,\varrho(\theta_j))$. Again, one easily sees that $\varrho_\Delta(s)$ converges to $\varrho(s)$ as $\Delta\searrow 0$. Also, for small $\Delta$, the functions $\varrho_\Delta(s)$ have support in a fixed compact interval and are uniformly bounded by the maximum value of $\varrho(s)$. If we insert measures $d\nu_\Delta(s)=\varrho_\Delta(s)d\nu(s)$ into (\ref{eqa:LevyFormulaII}), this expression converges to the related one with $d\nu(s)=\varrho(s)ds$. This  suffices to prove that an approximation of the probability distributions (in law) is feasible with 1st order spline densities.

\paragraph{Fixing drift and diffusion terms.} One might or might not like to include the drift and diffusion term determined by $b$ and $\sigma^2$ into the estimation problem. Although, in general, these quantities have to be estimated, here we keep  a small fixed value for $\sigma^2$ for reasons of numerical stability of the Kolmogorov equations. As long as this value underestimates the true diffusion, this corresponds to a splitting of the L\'evy process $Y(t)=Z(t)+L(t)$ into stochastically independent components where $Z(t)$ is determined by the purely Gaussian L\'evy triplet $(b,\sigma^2,0)$. $L(t)$ will be a L\'evy process that contains drift, the excess diffusion and jumps. However, the distribution of $L(t)$ at time $T$ can be approximated in the  sense of convergence in law by compound Poisson distributions without drift and L\'evy terms. The explicit construction can e.g. be found in \cite{BF}. We can thus approximate our estimation problem to arbitrary accuracy with a problem where drift and diffusion take fixed values. We also note that a non vanishing diffusion implies the existence of a probability density function $f(x,t)$ for the distribution of $Y(t)$.

\paragraph{Periodic boundary conditions from the projection to a torus.} Another problem with the Kolmogorov forward  equation is the issue of boundary conditions. We have already shown that we can approximate the estimation problem by one where the L\'evy measure $\nu$ has support inside a large interval $[\Omega_a,\Omega_b]$. Let $\Omega$ be the torus $[\Omega_a,\Omega_b]$ with the end points of the interval identified. Let $K=\Omega_b-\Omega_a$, then define the operator $+_K$ as $x+_K y:=(x+y)\mod (K)$ a group operation on $\Omega$, with $(\cdot) \mod (\cdot)$ the modulus operation. 
Let furthermore 
\begin{equation}\label{eq:homo}
\phi:(\mathbb{R},+)\to (\Omega,+_K)
\end{equation}
be the group homomorphism defined by $\phi(x)=(x)\mod (K)$.  Let $X(t)=\phi(Y(t))$ be a stochastic process on $\Omega$. If $Y(t)$ is a L\'evy process on the group $(\mathbb{R},+)$, the same applies to $X(t)$ with respect to the group $(\Omega,+_K)$.  Note that in the definition of L\'evy processes only the group structure of $(\mathbb{R},+)$ enters. L\'evy processes are naturally defined on locally compact Abelian groups like $(\mathbb{R},+)$ or also $(\Omega,+_K)$, see \cite{BF}. Let us now consider the characteristic function of the $\Omega$-valued process $X(t)$ at time $t$. By the periodicity of $\Omega$, only $k$ values from $\frac{2\pi}{K}\mathbb{Z}$ are needed. To derive a L\'evy - Kinchine formula (\ref{eqa:LevyFormula}) for $X(t)$ on $\Omega$ from that of $Y(t)$ on $\mathbb{R}$, we consider for such values of $k$
\begin{eqnarray}
\mathbb{E}\left[e^{iX(t)k} \right]&=&\mathbb{E}\left[e^{i\phi(Y(t))k} \right]\nonumber\\
&=&\mathbb{E}\left[e^{iY(t)k} \right] =e^{t\psi(k)}.
\end{eqnarray}
Inserting (\ref{eqa:LevyFormulaII}), we obtain
\begin{eqnarray}
\label{eqa:wrappedLevyMeasure}
\psi(k)&=& ibk-\frac{\sigma^2}{2}k^2+\int_{\mathbb{R}\setminus \{0\}}\left(e^{isk}-1\right)d\nu(s)\nonumber\\
&=& ibk-\frac{\sigma^2}{2}k^2+\int_{\mathbb{R}\setminus \{0\}}\left(e^{i\phi(s)k}-1\right)d\nu(s)\\
&=& ibk-\frac{\sigma^2}{2}k^2+\int_{\Omega\setminus \{0\}}\left(e^{isk}-1\right)d\phi_*\nu(s),\nonumber
\end{eqnarray}
with $\phi_*\nu$ the image measure on $\nu$ under $\phi$.   Note that under the hypothesis that $\nu$ has support in $[\Omega_a,\Omega_b]$, $\nu$ can be reconstructed from $\phi_*\nu$ as $\nu=\phi^{-1}_*\phi_*\nu$, where $\phi^{-1}:\Omega\to\mathbb{R}$ is the natural embedding. Using this, we identify $\nu$ and $\phi_*\nu$ in the following. 

Summing up, we consider the maximum likelihood parameter estimation problem with fixed $b,\sigma^2$, a piecewise linear density function with $N_\Theta$ grid points for the finite L\'evy measure and periodic boundary conditions on $\Omega$. By refining the grid for the linear interpolation, enlarging the size of the torus and letting the fixed diffusion go to zero, we can approximate the distribution of any L\'evy process with one of our candidate processes in the topology set by weak convergence in law. This constitutes the hierarchy of maximum likelihood estimation procedures.

\section{\label{sec:Kol}Kolmogorov Equations and Optimality for the Log-Likelihood}

In this section we formulate the optimization problem for the maximum likelihood parameters
estimation. The maximum likelihood estimator as an optimizing objective functional is given
together to the Kolmogorov forwards PIDE as constraints. The optimality system is written
by using the Lagrange multipliers method in a functional space, by also including the Karush-Kuhn-Tucker
conditions for the non negativity of the optimizing parameters.

\paragraph{Objective functional and forward equation.} Let the $L$ independent sample values $X_1,\ldots, X_L$ be given, and $X_l \in\Om, \quad l=1,\ldots,L$, where
$\Om=[\Om_a,\Om_b]$. These values can e.g. be obtained as $X_l=\phi(Y_l)$, where 
$\phi(\cdot)$ is the group homomorphism defined in Eq. (\ref{eq:homo}) and 
$Y_l$ is the L\'evy process on $\mathbb{R}$.
We deal with the problem to find the PDF  of $X(T)$ such that it
best fits with the sample values. For this purpose we consider the maximum likelihood problem in the framework of PIDE-constrained 
optimization: We have to find the maximum likelihood estimator 
\begin{equation}\label{eq:max_likeli}
 \max_{\bar\alpha} J(f,\bar\alpha),
\end{equation}
with respect to the parametrization of the measure given by $\bar\alpha$, where 
\begin{equation}\label{eq:functional}
 J(f,\bar\alpha) = \frac{1}{L} \sum_{l=1}^L \log(f(X_l,T,\bar\alpha)),
\end{equation}
is the (normalized) log-Likelihood with the constraint given by the following Kolmogorov forward (Fokker-Planck) equation for the L\'evy process $X(t)$ on the torus $\Omega$ with L\'evy data $(b,\sigma^2,\nu_{\bar \alpha})$ using the parametrization (\ref{eqa:LevyFormulaII}) and $d\nu_{\bar \alpha}(s)=\sum_{j=1}^{\Nth} \alpha_j\Theta_j(s)ds$:
\begin{equation}\label{eq:fp}
\left\{
\begin{array}{l}
\de_t f(x,t) + b \de_x f(x,t) -\frac{\sigma^2}{2}\de_x^2 f(x,t) -\int_{\Om}
\sum_{j=1}^{\Nth} \alpha_j (f(x+s,t)-f(x,t))\Th_j(s) ds  = 0  \\ \\
f(x,0)=f_0(x)   \\ \\
f(\Om_a,t) = f(\Om_b,t), \qquad \de_x f(\Om_a,t) = \de_x f(\Om_b,t), \\
\end{array}\right.
\end{equation}
where $f(x,t)$ represent the PDF of the process at time $t$.
This PIDE is defined in the interval of time $t\in [0,T]$, and with periodic boundary 
conditions on $\Om_a = \Om_b$. 
Here $\Theta_j(s)$ is a set of triangular shaped basis for the set of continuous functions that are linear on $(\theta_{j-1},\theta_j)$, see the preceding section,
\begin{align}
 \Th_j(s)=1+(s-\theta_j)/\De \qquad  s\in [\theta_j - \De, \theta_j]  \nonumber \\
 \Th_j(s)=1-(s-\theta_j)/\De \qquad  s\in [\theta_j, \theta_j+\De], \nonumber
\end{align}
where $\theta_j = \Om_a + \De (j-1)$ for $j=1,\ldots \Nth+1$, are the points of a discrete
uniform mesh of step size $\De=(\Om_b-\Om_a)/\Nth$ defined on the domain.
The periodicity $\Th_1(s)\equiv\Th_{\Nth+1}(s)$ is assumed.

The existence and uniqueness of the solution of the Fokker-Planck equation (\ref{eq:fp}) is well established, 
also for initial conditions belonging to the class of measures \cite{BF}.

\paragraph{The optimality system.} If we write the mapping $\bar\alpha\rightarrow f(\bar\alpha)$
between the maximization parameters and the PDF, then we introduce the so-called
reduced cost functional $\hat J(\bar\alpha)=J(f(\bar\alpha),\bar\alpha)$, so that the
maximization problem becomes
\begin{equation}\label{eq:max_reduced}
\max_{\bar\alpha} \hat J(\bar\alpha) =  \max_{\bar\alpha} J(f(\bar\alpha),\bar\alpha)
\end{equation}
A local maxima $\bar\alpha^*$ for $\hat J$ can be found by solving
the optimality system obtained by vanishing the variations of the following Lagrangian functional
\begin{align}\label{eq:Lagr}
 \eL(f,p,\bar\alpha,\bar \pi)  = & \frac{1}{L} \sum_{l=1}^L \log(f(X_l,T,\bar\alpha))
+ \intt \!\! \intx \bigg[ \de_t f(x,t) + b \de_x f(x,t) - \half{\sigs} \de_x^2 f(x,t) \nonumber \\
&  - \sum_{j=1}^{N_\Theta} \alpha_j \intx (f(x+s,t)-f(x,t))\Th_j(s) ds \bigg] p(x,t)  dx\,dt-\sum_{j=1}^{N_\Theta}\pi_j\alpha_j,
\end{align}
where $\bar\pi=(\pi_1,\ldots,\pi_{N_\Theta})$ and $\bar \alpha$ fulfil the usual Karush-Kuhn-Tucker (KKT) conditions $\pi_j\alpha_j=0$ and $\pi_j\geq 0$.
These are important to include the non-negativity constraints for the control variables. Note that if the condition
$\alpha_j \geq 0$ is violated for some $j\in\{1,\ldots,N_{\Theta}\}$, the density of the measure $d\nu_{\bar\alpha}(s)$ is negative in a neighbourhood of $\theta_j$ and thus is not a L\'evy measure any more. The sum $\sum_{j=1}^{N_\Theta}\pi_j\alpha_j$ should be 
extended only on the active constraints, i.e. when $\alpha^*_{j'} = 0$. For those
values of $j$ on the maximum where $\alpha^*_{j}>0$ we have $\pi^*_{j}=0$.

First we calculate the variation $\eL(f+\dl f)-\eL(f)$ for the adjoint equation.
In the following the variations are calculated separately for each addend
of the r.h.s. We get
\begin{align}
- \sum_{j=1}^{\Nth} \alpha_j \intt\intx\intx (-\dl f(x,t))\Th_j(s)p(x,t) ds\, dx\,dt \nonumber \\
 =\sum_{j=1}^{\Nth} \alpha_j \intt\intx \Big( \intx \Th_j(s) ds\Big) \dl f(x,t) p(x,t)\,dx\,dt.
\end{align}
For the term 
$- \sum_j \alpha_j \intt\intx\big(\intx \dl f(x+s,t)\Th_j(s) ds\big) p(x,t)\, dx\,dt$
we apply the substitution $y=x+s$, then exchange $x \leftrightarrow y$,
so that it recasts to $- \sum_j \alpha_j \intt\intx\big(\intx p(y,t) \Th_j(x-y) dy\big) 
\dl f(x,t) dx\,dt$. Then, again, we substitute $s=x-y$ and, by inserting also the former term, we get
\begin{equation}
- \intt\!\!\intx \Big[ \sum_{j=1}^{\Nth} \alpha_j \intx ( p(x-s,t) -p(x)) \Th_j(s) ds \Big]
\dl f(x,t) dx\,dt.
\end{equation}
For the variation of the time derivative, integrating by parts, one obtains
\begin{equation}
\intx \dl f(x,t)\, p(x,t)|_0^T dx -\intt \de_t p(x,t) \, \dl f(x,t) dx dt.
\end{equation}
The variation
$\dl f(x,0)=0$ holds because of the Cauchy initial condition, while the variation in $T$ can be
defined in some points $X_l$.
Next, we integrate by parts the term with the first order derivative in $x$ and obtain
$ b \intt ( \dl f(x,t) p(x,t))|_{\Om_a}^{\Om_b} dt - b\intx \de_x p \,\dl f dx\, dt$.
Due to  periodicity in the first term, $\dl f(\Om_a,t)(p(\Om_b,t)-p(\Om_a,t))$
has to be zero, hence $p(\Om_b,t)=p(\Om_a,t)$.

From the diffusive term we get
 \begin{equation}
 -\half{\sigs} \intt\!\!\intx \de_x^2 \dl f p(x,t) dx dt=
-\half{\sigs} \intt [\de_x \dl f p - \dl f \de_x p]_{\Om_a}^{\Om_b} dt +
\intx \de_x^2 p(x,t) dx dt.
\end{equation}
The first boundary term is zero because of 
the periodic condition of the variation of the derivative of $f$ at the boundaries, and because of 
the previous periodic condition on $p$. The second is analogous and
has to vanish, we therefore get the continuity condition $\de_x p(\Om_a,t)=\de_x p(\Om_b,t)$.

By collecting all the terms under double integral, we get the adjoint equation.
The remaining boundary term $\intx \dl f(x,T) p(x,T) dx$ will be considered below.

To calculate the variation on $f$ in the functional $J$ we perform an additional
integration in space, so that
\begin{equation}\label{eq:int_func}
\dfrac{1}{L} \sum_{l=1}^L \intx \log(f(x,T)) \dl(x-X_l) dx,
\end{equation}
where $\dl(.)$ is the $\dl-$Dirac measure, then variate $f(x,T)+\dl f(x,T)$, hence 
\begin{align}
\intx \log(f(x,T)+\dl f(x,T))\dl(x-X_l) dx \nonumber\\
= \intx (\log(f(x,T))+\dl f(x,T) / f(x,T))\dl(x-X_l) dx,
\end{align}
 so that the first order terms plus the remaining boundary, give
\begin{equation}
\dfrac{1}{L} \sum_{l=1}^L \intx \dfrac{\dl f(x,T)}{f(x,T)}\dl(x-X_l) dx +
\intx p(x,T) \dl f(x,T) dx.
\end{equation}
This expression have to be zero for each $\dl f(x,T)$. It represents the terminal
condition for the adjoint equation: that is $p(X_l,T)=-1/(L f(x_l,T))$, and
 $p(x,T)=0,$ if $x\neq \{X_1,\ldots,X_L\}$. In case of multiplicity of
 $X_l$ the condition becomes $p(X_l,T)=-1/(L \sum_{l'} f(X_{l'},T))$, with $l'$ running on
the multiplicity value.

Summarizing, the adjoint equation (Kolmogorov's backward equation) is as follows:
\begin{equation}\label{eq:adj}
\left\{
\begin{array}{l}
-\de_t p(x,t) - b\, \de_x p(x,t) -\dfrac{\sigma^2}{2}\de_x^2 p(x,t) -\int_{\Om}
\sum_{j=1}^{\Nth} \alpha_j (p(x-s,t)-p(x,t))\Theta_j(s) ds  = 0  \\
\\
p(x,T)=-\frac{1}{L}\sum_{l=1}^L \delta(x-X_l)/f(X_l,T) \\
\\
p(\Om_a,t)=p(\Om_b,t),\quad  \de_x p(\Om_a,t)=\de_x p(\Om_b,t).
\end{array}\right.
\end{equation}

We note that by reverting the sign of the time we get the same PIDE as the
forward equation (up to a reflection of the drift and jump direction), hence  this equation has a unique solution, also for the non regular final value problem \cite{BF}.

Second, we variate in Eq.(\ref{eq:Lagr}) the fitting parameters $\eL(\alpha_j+\dl \alpha_j)-\eL(\alpha_j)$,
from which we found the optimality equations:
\begin{equation}\label{eq:opt}
-\pi_{j'} -\int_0^T \!\! \int_\Om \int_\Om (f(x+s,t)-f(x,t)) p(x,t) \Theta_j(s)\, ds\, dx\, dt = 0, \qquad j=1,\ldots,\Nth,
\end{equation}
where $j'$ runs on the set of values where $\alpha^*_{j'}=0$. Note that the active $\pi_{j'}$ do not change the gradient, but simply balance non-zero gradient components that point to the directions where the inequality constraint $\alpha_{j'}\geq 0$ is violated. As in our case we deal with simple box-constraints on the $\alpha_j$ themselves, we can set those components of the negative gradient equal to zero that correspond to an active index $j'$ and are negative, when determining the update. This then accounts for the effect of the $\pi_{j'}$, see e.g. \cite{Fri,nocedal}.

The $1^{\mbox{st}}$ order necessary optimality system consists of the Eqs. (\ref{eq:fp}), (\ref{eq:adj}) and (\ref{eq:opt}).
Its solution gives values $\alpha_1^*, \ldots, \alpha_{\Nth}^*$ that are candidates for maximizing the functional (\ref{eq:functional}). Note that maximum likelihood fits in most cases do not correspond to convex optimization problems 
and one always has to account for the perils of local minima that are sub-optimal globally.

\paragraph{Forward equation in flux form.} The forward equation (\ref{eq:fp}), 
can be written in flux form:
$\de_t f(x,t) = \de_x \eF(x,t)$, where $\eF(x,t)$ is the flux defined as
\begin{equation}\label{eq:flux}
\eF(x,t) = - b f(x,t) + \frac{\sigma^2}{2}\de_x f(x,t) + \sum_{j=1}^{\Nth} \alpha_j\int_{\Om}
\Big(\int_{-s}^0 f(x-y,t) dy\Big)  \Theta_j(s) ds .
\end{equation}
By using 
$\de_x \int_{-s}^0 f(x-y) dy = \int_{-s}^0 f'(x-y) dy = \int_x^{x+s} f'(z) dz = f(x+s)-f(x)$,
it is easy to verify that Eq. (\ref{eq:flux}) is equivalent to Eq. (\ref{eq:fp}).
Further, from the conservation of the total probability, it follows that the flux has the periodic boundary 
condition $\eF(\Om_a,t)=\eF(\Om_b,t)$. From this we immediately  get the periodic condition on the 
first derivative $\de_x f(\Om_a,t) = \de_x f(\Om_b,t)$.

\section{\label{sec:Num} Numerical Scheme}

The numerical solution of the optimality system is found by a non linear gradient 
conjugate iterative procedure \cite{gilbert,shanno,mar:alf}. 
At each iteration the solution of two PIDEs, the forward
and the adjoint one, must be found. In particular the structural properties of the PDF
solution must be satisfied, as well as a stability condition of the PIDEs numerical scheme solver.

For the numerical discretization of the Kolmogorov forward equation we use the Chang-Cooper
scheme (CC) \cite{chacoo}, joint to a 2nd order backward differentiation formula
(BDF2) for the discrete time operator.
The CC method was proposed for a Fokker-Planck resp.\ Kolmogorov equation \cite{mar:alf}
without the integral term. It is  stable, second-order accurate, non-negative, and conservative numerical scheme  \cite{mar:alf,bor:moh}.

The CC method is used for the differential operators, the integral term is treated separately according to an IMEX methodology.
We denote the following $B=-b$ and $C={\sigma^2}/{2}$, then the Kolmogorov forward
equation reads as follows
\begin{equation}\label{eq:forw}
\de_t f(x,t) = \de_x F(x,t) + \int_{\Om}
\sum_{j=1}^{\Nth} \alpha_j (f(x+s,t)-f(x,t))\Th_j(s) ds,
\end{equation}
where
\begin{equation}
 \label{eq:def_flux}
F(x,t) = B\, f(x,t) + C \de_x f(x,t).
\end{equation}
Consider a uniform grid of size $h$ on the space domain $\{ \Omega_h \}_{h>0}$ given by
$\Omega_h = \{ x \in \R : x_i = i\,h +\Om_a, i=0,\ldots,N, h=(\Om_b-\Om_a)/N \}$ and
a uniform grid on the time domain
$I_{\delta t} = \{ t \in [0,T] : t_m = m\,\delta t ,   m=0,\ldots,N_T, \dl t=T/N_T\}$. Let
$f_i^m \approx f(x_i,t_m)$ denote the approximated values of the continuous solution of the FPE.
We employ the following discretization of (\ref{eq:flux})
\begin{equation}
\label{eq:flux_eq}
\partial_{BD}^- f^{m+1}_i = \frac{1}{h}  (F_{i+1/2}^{m+1} - F_{i-1/2}^{m+1}) + Q( f^m_i;\bar\alpha),
\end{equation}
where
\begin{equation}
\partial_{BD}^- f^m_i = \dfrac{3 f^m_i -4 f^{m-1}_i + f^{m-2}_i}{2 \delta t},
\end{equation}
is the BDF2 operator.
$Q(f^m_i;\bar\alpha)$ is the sum of the integrals of Eq. (\ref{eq:flux}) calculated with 
the mid-point scheme
\begin{equation}\label{eq:defQ}
Q(f^m_i;\bar\alpha) = h \sum_{j=1}^{\Nth} \alpha_j \sum_{k=1}^N \hat f_{ik}^m \theta_{jk} - a f_i^m,
\end{equation}
where $\bar\alpha=(\alpha_1,\ldots,\alpha_{\Nth})$, $\hat f_{ik}^m \approx ( f(x_i+s_k,t_m) )_{\Omega_h}$
represents the translated $f_{i}^m$ by the value $s_k\in\Omega_h$ and continued by periodicity, 
$\theta_{jk} = \Theta_j(s_k)$,
\begin{equation}\label{eq:const_a}
a = h \sum_{j=1}^{\Nth} \alpha_j \sum_{k=1}^N \theta_{jk}
\approx \sum_{j=1}^{\Nth} \alpha_j \int_{\Om} \Th_j(s) \, ds.
\end{equation}
Note also that the summation starts 
from $k=1$, because the point $k=0$ is the same of that $k=N$.
Therefore, the solution at a new time step is calculated by solving the 
following equation for the unknown $f^{m+1}_i$
\begin{equation}
\label{eq:disc_fp}
3f^{m+1}_i - \frac{2\dt}{h}  (F_{i+1/2}^{m+1} - F_{i-1/2}^{m+1}) = 4f^{m}_i - f^{m-1}_i 
+ \dt\, Q(f^m_i;\bar\alpha), 
 \end{equation}
with the initial condition 
\begin{equation}\label{eq:disc_init}
  f^{0}_i=f_{0,i} .
\end{equation}

This scheme is based on the fluxes at $N$ cell boundaries. 
The partial flux  at the position $x_{i+h/2}$ is computed as follows
\begin{equation}
\label{eq:CCflux}
 F_{i+1/2}^{m+1} = \left[ (1-\delta) B + \frac{1}{h} C\right] \, 
f^{m+1}_{i+1} - \left(\frac{1}{h} C - \delta B \right)\, f^{m+1}_{i} .
\end{equation}

This formula results from the following linear convex combination of $f$ at the 
points $i$ and ${i+1}$:
\begin{equation}
f^{m+1}_{i+1/2} =(1-\delta)\, f^{m+1}_{i+1} + \delta \, f^{m+1}_{i}, \qquad \delta \in [0,1].
\end{equation}
The idea of implementing this combination was proposed by Chang and Cooper in \cite{chacoo} and it was 
motivated with the need to guarantee positive solutions that preserve the equilibrium configuration. Indeed, the CC method is related to exponential fitting methods, such as that one proposed 
by Allen and Southwell \cite{all:sou}, and by the  Scharfetter-Gummel discretization scheme
\cite{sch:gum}.
The value of the parameter $\delta$ is $\delta = 1/w - 1/(\exp(w)-1)$,
where $w = h \, B/C$, which can be shown to be 
monotonically decreasing from $1$ to $0$ as $w$ goes from $-\infty$ to $\infty$.
Notice that with the choice of $\delta$ given above, the numerical 
scheme shares the same properties of the continuous FP equation that 
guarantee positiveness and conservativeness. This is a special case of the CC scheme
because in the general one, the functions $B$ and $C$ may depends on $(x,t)$, hence also 
$\delta$ may depend on $(x,t)$, too. Both the CC scheme \cite{bor:moh} and the mid-point are second order
accurate, then  a second order numerical scheme results.

Let $f^m=(f^m_1,\ldots, f^m_{N})^\dag$ be the discrete solution at the time $t_m$,
with $f_0^m$ omitted due to periodicity,
and $\beta=C/h-\delta B$. The action of the finite difference operator for $F^m$ in Eq. (\ref{eq:flux_eq}) reads as matrix $A$ whose elements are defined by
\begin{equation}\label{eq:diff_flux_matrix}
A_{i,i}=-\beta(1+\omega)/h, \quad A_{i, i-1}=\beta/h, \quad A_{i, i+1}=\omega\beta/h, 
\quad A_{1,N}=\beta/h, \quad A_{N,1}=\omega\beta/h,
\end{equation}
where $\beta=B/(\omega-1)$, $\omega=\exp(hB/C)$. Hence,
$  A f^m := (F_{i+1/2}^m-F_{i-1/2}^m)/h$, and then the
Eq. (\ref{eq:disc_fp}) can be written in matrix form, as follows
\begin{eqnarray}\label{eq:discFP_matrix}
& M f^{m+1}  & = 4f^{m} - f^{m-1} + \dt\, Q(f^m;\bar\alpha), \\
& \mbox{where  } M &: = 3I-2 \,\delta t\,A
\end{eqnarray}
is the matrix coefficients related to Eqs. (\ref{eq:disc_fp}) and 
(\ref{eq:CCflux}).
We note that this method needs of a second starting point, that can be calculated by using a first order Euler scheme with a smaller time step size than $\delta t$.

The implicit Euler scheme for the Eqs. (\ref{eq:forw}) and (\ref{eq:CCflux}) is
\begin{equation}\label{eq:Euler}
(I-\delta t A)f^m=f^{m-1}+ \dt\, Q(f^{m-1};\bar\alpha).
\end{equation}

These two numerical schemes own some properties that can be easily
proved, but we list here as remarks. 
\begin{rem} \label{rem:eul_cons}
The Euler-CC scheme (\ref{eq:Euler}) to Eqs. (\ref{eq:forw}) and (\ref{eq:CCflux}),
defined in the periodic domain $\Omega_h$, is conservative.

In fact, $\sum_{i=1}^N A_{i,j}=0,\; \forall j$, and $\sum_{i=1}^N Q(f_i^m,\bar\alpha)=0$ because the set of values of $f_i^m$
 are the same as $\hat f_{ik}^m$, being the last only translated by $k$.
Hence, $\sum_{i=1} f^m_i = \sum_{i=1} f^{m-1}_i $. 
\end{rem}

\begin{rem}\label{rem:cons_bdf}
Provided that   $\sum_{i=1} f^m_i = \sum_{i=1} f^{m-1}_i $, then
the BDF2-CC scheme (\ref{eq:discFP_matrix}) to Eqs. (\ref{eq:forw}) and (\ref{eq:CCflux}),
defined on the periodic domain $\Omega_h$, is conservative. 
In fact for the same constraints on $A$ and $Q$ as above, we get the identity
$3 \sum_{i=1} f^{m+1}_i = 4 \sum_{i=1} f^{m}_i - \sum_{i=1} f^{m-1}_i
= 3 \sum_{i=1} f^{m}_i$.
\end{rem}
The positivity of the numerical scheme is proved by using the theorem
for the class of $M$-matrix \cite{ple}. Given a positive matrix $E$, $E_{ij}\geq 0$,
we say that $M=sI-E$ is a non singular $M$-matrix if $s>\rho(E)$, where 
$\rho(E)$ is the spectral radius of $E$. 
A non singular $M$-matrix has the important property 
\begin{equation}\label{eq:mmatrix}
\mbox{$M$ is non singular $M$-matrix} \Rightarrow M^{-1}\geq 0.
\end{equation}
\begin{thm}\label{th:pos_eul}
Let $\delta t \leq 1/a$, with $a$ defined in (\ref{eq:const_a}), then the Euler scheme (\ref{eq:Euler}) to Eq. (\ref{eq:forw}), defined in the periodic domain $\Omega_h$, is positive preserving.
\end{thm}
\begin{proof} The argument is as follows: let $R$ the matrix operator such that\\ 
$R f^m = h\sum_{j=1}^N \alpha_j \sum_{k=1}^N \theta_{jk} \hat f_{ik}^m$. Such a matrix
is non negative because $\alpha_j$ and $\theta_{jk}$ are.
The numerical scheme (\ref{eq:Euler}) can be recast as
\begin{equation}\label{eq:eul_proof}
\left(\left(1+\frac{\delta t \beta}{h}(1+\omega)\right) I-\delta t \tilde A\right)f^m=(1-a\,\delta t)f^{m-1}+\delta t R f^{m-1},
\end{equation}
where $\tilde A=A-\mbox{diag}(A)$ is a positive matrix. 
Provided that $f^{m-1}\geq 0$ and $ \delta t \leq 1/a$ the r.h.s. is a non negative vector.
We observe that the matrix on the l.h.s is always diagonal dominant, hence it has a convergent regular splitting and consequently is an $M$-matrix \cite{ple}.
Therefore, $(I-\delta t A)^{-1}$ is non negative and $f^m$ will be too. 
\end{proof}

In order to prove the positivity of the BDF2 numerical scheme (\ref{eq:discFP_matrix}),
we need of the following Lemma that gives a lower bound to the velocity of decreasing of the solution.
\begin{lem}\label{lem:eul}
Let the vector $f^m \in \R^N$ be given non negative. Take a number
$\xi>1$, then the solution $f^{m+q}$ calculated with the Euler scheme of Eq. (\ref{eq:Euler}) after $q$ time steps satisfies the following inequality
$$
f^{m+q} \geq f^m/\xi^{q},
$$ 
provided that $\dt < \dfrac{\xi-1}{a \xi+\beta(1+\omega)/h}$, with 
parameters defined in Eq. (\ref{eq:diff_flux_matrix}).
\end{lem}
\begin{proof} A proof is given for a particular case in \cite{bor:moh} (see also Refs. therein).
Here we prove it as follows.
Given $f^m$ and $f^{m+1}$ calculated with (\ref{eq:Euler}), let define 
$v=\xi f^{m+1}-f^m$. By applying the operator $I-\dt A$, we get
$(I-\dt A) v = (\xi I-(I-\dt A))f^m + \dt \xi Q(f^m;\bar\alpha)$, i.e.
$$ 
(I-\dt A) v= ((\xi-1-\dt (\beta(1+\omega)/h+a\xi))I+\dt\tilde{A}+\dt\xi R)f^m,
$$
where $\tilde A=A-\mbox{diag}(A)$ is a positive matrix. Now provided the bound 
for $\dt$, then the r.h.s. is positive and from Th. \ref{th:pos_eul} we get that
$v\geq 0$. By iterating that inequality $q$ times, we get the thesis. 
\end{proof}

\begin{rem}\label{rem:bound_dt_eul}
The upper bound on $\dt$ in Lemma \ref{lem:eul} results to be $\dt < 1/a$ for $\xi>1$,
hence the condition on the Lemma is stricter than those on non negativity of
Thm. \ref{th:pos_eul}.
\end{rem}

Now we show a Lemma similar to Lemma \ref{lem:eul} valid for the BDF2 scheme.

\begin{lem}\label{lem:bdf2}
Let $1<\xi<3$ and $\dt\leq h(\xi-1)(3-\xi)/(a\xi h+2\beta(1+\omega))$ be the time
step size of the numerical scheme of Eq. (\ref{eq:discFP_matrix}) that generates
the sequence of vectors $f^m$ for $m=2,3,...$ from the starting vectors $f^0,f^1$.
If there exists $m^*$ such that $\xi f^{m^*+1}-f^{m^*}\geq 0$ and $f^{m^*}\geq 0$, then 
$\xi f^{m+1}-f^{m}\geq 0$ for all $m>m^*$.
\end{lem}
\begin{proof}
We apply the operator $(3I-2\dt A)$ to $v=\xi f^{m+2}-f^{m+1}$,
$$
(3I-2\dt A)v=\xi (3I-2\dt A)f^{m+2}-(3I-2\dt A)f^{m+1}
$$
and use Eq. (\ref{eq:discFP_matrix}) to the first term on the r.h.s.\ to get
$$
(3I-2\dt A)v=[4\xi-3-\dt(a\xi+2\beta(1+\omega)/h)]f^{m+1} -\xi f^m + \dt(2\tilde A +\xi R)f^{m+1},
$$
where $\tilde A=A-\mbox{diag}(A)$ is a positive matrix. 
We know that $(3I-2\dt A)$ is an $M$-matrix and its inverse is always non-negative. Also
$2\tilde A +\xi R$ is non negative. Hence, we can prove non negativity of $v$,
provided that 
$$
4\xi-3-\dt(a\xi+2\beta(1+\omega)/h)\geq \xi^2,
$$
for a value $m=m^*$, because of the hypothesis $\xi f^{m^*+1}-f^{m^*}\geq 0$, that
also states that $f^{m^*+1}\geq 0$.
The last inequality is just the bound on $\dt$ in the assertion that gives a
positive value for $\dt$ only when $1<\xi <3$. 
\end{proof}

Indeed, this Lemma proves positivity of the numerical solution of Eq. (\ref{eq:discFP_matrix}),
provided that $f^0\geq 0$, and $\xi f^{1}-f^{0}\geq 0$. $f^{1}$ is the second starting value
of the numerical scheme, that can be calculated with the Euler scheme (\ref{eq:Euler}).

\begin{thm}\label{th:pos_bdf2}
Let $f^0\geq 0$ the discrete initial condition (\ref{eq:disc_init}), and let
$f^1$ the second starting value calculated with the Euler scheme (\ref{eq:Euler}) with 
an appropriate time step, such that $\xi f^1-f^0 \geq 0$, for $1<\xi <3$. Then, the BDF2 scheme (\ref{eq:discFP_matrix}) to Eq. (\ref{eq:forw}), defined in the periodic domain $\Omega_h$, is positive preserving for the solution $f^m$, with $m>1$.
\end{thm}
\begin{proof}
 The proof is an application of the Lemmas \ref{lem:eul} and \ref{lem:bdf2}.
\end{proof}

In order to establish the stability of the discrete numerical schemes
of Eqs. (\ref{eq:discFP_matrix}) and (\ref{eq:Euler}), we need inequalities
of the form $\Vert f^{m+1} \Vert \leq K \Vert f^{m} \Vert$ evaluated 
in a suitable norm with $K$ possibly less or equal than $1$.
We prove that it realizes for the $1$-norm with $K=1$.

\begin{thm}\label{th:stab_eul}
Let the positivity condition of Theorem \ref{th:pos_eul} be fulfilled,
i.e. $\dt \leq 1/a$. Then, the Euler scheme (\ref{eq:Euler}) is stable
in the $1$-norm, that is  $\Vert f^{m} \Vert_1 \leq \Vert f^{m-1} \Vert_1$
for all $m$.
\end{thm}
\begin{proof}
Let $r=\dt\, \beta/h$ and invert the matrix operator at l.h.s., then Eq. (\ref{eq:eul_proof}) reads as
$$
f^m= \dfrac{\left(I-\dfrac{r\tilde A}{1+r(1+\omega)}\right)^{-1}}{1+r(1+\omega)}  [(1-a\,\delta t)f^{m-1}+\delta t R f^{m-1}].
$$
Now we observe that 
$$
\left\Vert \left(I-\dfrac{r\tilde A}{1+r(1+\omega)}\right)^{-1} \right\Vert_1\leq
\left(1-\left \Vert \dfrac{r\tilde A}{1+r(1+\omega)}\right\Vert_1 \right)^{-1}=
1+r(1+\omega).
$$
Hence, 
$$
\Vert f^m \Vert_1 \leq \Vert (1-a\,\delta t)f^{m-1}+\dt\, R f^{m-1}\Vert_1.
$$
Since $\dt \leq 1/a$, all the components of the vectors inside the norm
at the r.h.s. are positive, so that the modulus for the evaluation of the $1$-norm can 
be removed. Using $\sum_{i=1}^N Q(f_i^m,\bar\alpha)=0$ as in Rem. \ref{rem:eul_cons}, we get the statement of the theorem.
\end{proof}

Now we can prove the stability of the numerical scheme with BDF2 integration 
of Eq. (\ref{eq:discFP_matrix}).

\begin{thm}\label{th:stab_bdf2}
Let the positivity condition of the Theorem \ref{th:pos_bdf2} be fulfilled, i.e.
let $\dt$ be the time step size of the numerical scheme of Eq. (\ref{eq:discFP_matrix}),
$f^0\geq 0$ the discrete initial condition (\ref{eq:disc_init}) and 
$f^1$ the second starting value evaluated at the time $\dt$. 
If there exists a real number $\xi$ such that $f^1\geq f^0/\xi$ with $1<\xi <3$ and
$\dt \leq h(\xi-1)(3-\xi)/(a\xi h+2\beta(1+\omega))$, then the BDF2 scheme (\ref{eq:discFP_matrix}) is stable in the $1$-norm, that is  $\Vert f^{m} \Vert_1 \leq \Vert f^{m-1} \Vert_1$ for all $m$.
\end{thm}
\begin{proof}
The numerical scheme (\ref{eq:discFP_matrix}) can be written as 
$$
M f^{m+1} = (4-\zeta) f^m +(\zeta - a \dt) f^m - f^{m-1} + \dt R f^m,
$$
where $R$ is defined as in Thm. \ref{th:pos_eul}.
We apply $M^{-1}$ and evaluate the $1$-norm to both sides.
Following the same calculations as in Thm. \ref{th:stab_eul}, we get
that $\Vert M^{-1} \Vert_1 = 1/3$.

From the bound on $\dt$, we note that
\begin{equation}\label{eq:bound_bdf}
a\, \dt < (\xi-1)(3-\xi)/\xi \leq 4-2\sqrt{3} < 0.536.
\end{equation} 
This means that for all $\zeta$ in the interval $5-2\sqrt{3} < \zeta < 3$, it is
$\zeta - a\,\dt=\xi$ with $\xi \in (1,3)$. 
Now we have that $f^m\geq 0$ by virtue of the positivity 
condition, $ (\zeta - a \,\dt) f^m - f^{m-1}\geq 0 $ by our assumptions, 
and $4-\zeta > 0$, hence
is guaranteed that the sum in the r.h.s. is a non negative vector and
the modulus in the calculation of the $1$-norm can be removed.
Using the property given in Rem. \ref{rem:eul_cons}, we conclude that
$\Vert f^{m+1}\Vert_1 \leq  \Vert f^{m}\Vert_1$. 
\end{proof}

\begin{rem}\label{rem:stability}
Indeed, in the stability Theorem \ref{th:stab_bdf2} the equality $\Vert f^{m+1}\Vert_1 =  \Vert f^{m}\Vert_1$ holds. In fact, because of the conservativeness from Rem. \ref{rem:cons_bdf}
we have $\sum_{i=1} f^{m+1}_i = \sum_{i=1} f^{m}_i $, and under the non negativity
condition of Theorem \ref{th:pos_bdf2} all the components of the vectors 
$f^{m+1}_i, f^{m}_i$ are non negative, so that the previous conservativeness identity
corresponds to the 1-norm equivalence.
Further, we can state that for these numerical schemes the conservativeness
and the non negativity imply the stability of the discrete operator.
\end{rem}

\begin{rem}
We can finally conclude from the Lax equivalence theorem, that for regular solutions of the Kolmogorov forward equation $f(x,t), (x,t)\in [\Omega,T]$, provided that the hypothesis of Thm. \ref{th:stab_bdf2}, then
the numerical scheme of Eq. (\ref{eq:discFP_matrix}) yields
discrete solutions that are second order convergent in time and space.
\end{rem}

\begin{rem}
The non negativity conditions of for the Euler scheme of Lemma \ref{lem:eul}
and BDF2 of Lemma \ref{lem:bdf2}, can be correspondingly written as
$$
\dt \leq \dfrac{\xi-1}{a\,\xi+B \coth(h\,B/(2C))/h}
$$
and
$$
\dt \leq \dfrac{(\xi-1)(3-\xi)}{a\xi +2 B \coth(h\,B/(2C))/h}.
$$
We note that for $h\rightarrow 0$ or $B\rightarrow 0$ the upper bound
for $\dt$ scales as $h^2/C$. For $C\rightarrow 0$ it scales as $h/B$.
\end{rem}

%

\paragraph{Adjoint equation.}
The discrete adjoint equation can be found by discretizing the Lagrangian function
of Eq. (\ref{eq:Lagr}) and then performing the variations on the discrete
variables. This is know as the \textit{discretize-then-optimize} approach (see Ref. \cite{mar:alf} for details). This technique yields the following discrete adjoint equation
\begin{equation}\label{eq:discADJ_matrix}
 M^\dag p^{m} = 4p^{m+1} - p^{m+2} + \dt \tilde Q(p^{m+1};\bar\alpha), 
\end{equation}
where $M^\dag$ is the transpose of $M$, and 
$\tilde Q(p^{m+1};\bar\alpha) = h \sum_{j=1}^{\Nth} \alpha_j \sum_{k=1}^{N} \tilde p_{ik}^m \theta_{jk} - a p_i^m$,
with $\tilde p_{ik}^m \approx ( p(x_i - s_k,t_m) )_{\Omega_h}$.

The numerical stability is given by the same condition for the forward equation, since
the transpose of the operator $M$ has the same eigenvalues, but in this case the
non negativity and conservativeness property are not required.

Care has to be taken for the discrete terminal condition, since it can not be defined 
through the Eq. (\ref{eq:adj}) for the presence of the $\delta$-Dirac measure. 
For this purpose we discretize the term (\ref{eq:int_func}) as follows
$$
\dfrac{1}{L} \sum_{l=1}^L \sum_{i=1}^N \int_{x_i-1/2}^{x_i+1/2} \log(f(x,T)) \dl(x-X_l) dx=
\dfrac{1}{L} \sum_{l=1}^L \sum_{i=1}^N \log(f(\hat x_i,T)) 1_{\{X_l\in [ x_i-1/2,x_i+1/2 )\}},
$$
where $ \hat x_i$ are the points of the integral average theorem.
Then we use the approximation $f(\hat x_i,T)\approx f_{i}^{N_T}$, so that, by 
performing the variation $\delta f_i^{N_T}$ on this discrete functional, we get the discrete 
terminal condition
\begin{equation}\label{eq:disc_term}
 p^{N_T}_i = p_{T,i} = -\dfrac{1}{L} \sum_{l=1}^L 1_{\{X_l\in [ x_i-1/2,x_i+1/2 )\}} /f_{i}^{N_T}, \quad i=1,\dots,N.
\end{equation}
According to Eq. (\ref{eq:adj}), it completes the formulation of the discrete adjoint problem.

\paragraph{Discrete gradient.} The discrete of the reduced gradient related to the optimality condition of Eq. (\ref{eq:opt})
is calculated with the mid-point quadrature formula. Each component $j$ is given by
\begin{equation}\label{eq:disc_grad}
(D_{\bar\alpha} \hat J)_j := - \delta t \, h^2 \sum_{m=0}^{N_T} \sum_{i=1}^N \sum_{k=1}^N (\hat f^m_{ik} - f^m_i ) p_i^m  \theta_{jk},
\end{equation}
where $(D_{\bar\alpha} \hat J)_j \approx (\nabla_{\bar\alpha} \hat J)_j$.

\paragraph{Non linear conjugate gradient method.}

The availability of the discrete gradient allows us to implement a non linear conjugate gradient scheme (NLCG) in order to solve the optimization problem (\ref{eq:max_reduced}). 
NLCG represents an extension of the linear conjugate gradient method to
non-quadratic problems \cite{gilbert,shanno, mar:alf}.

The optimality system is solved by implementing the gradient given by the 
following algorithm:
\begin{algorithm}[Evaluation of the Gradient at $\bar\alpha$]
\label{evaluate_gradient}
 \noindent
\begin{enumerate}
 \item Solve the discrete FP equation (\ref{eq:disc_fp}) with given 
 initial condition (\ref{eq:disc_init});
 \item Solve the discrete adjoint FP equation (\ref{eq:discADJ_matrix}) 
 with terminal condition (\ref{eq:disc_term});
 \item Compute the approximated discrete gradient $D_{\bar\alpha} \hat J$ by using (\ref{eq:disc_grad});
 \item End.
\end{enumerate} 
\end{algorithm}
in a NLCG scheme. 
The search directions are recursively as
\begin{equation}\label{ncgdir}
d_{k+1}=-g_{k+1} + \beta_k \, d_{k}, 
\end{equation}
where $k=0,1,2,\ldots$ in this paragraph stands for the iteration index,
 $g_k= D_{\bar\alpha} \hat J(\bar\alpha_k)$ is the numerical gradient, with $d_0=-g_0$. 
Let $\bar\alpha_k$ an estimation of the best rates at the iteration $k$, 
the next one for a minimum point are given by
\begin{equation}\label{ncgupdate}
\bar\alpha_{k+1} = \bar\alpha_k + \xi_k\, d_k , 
\end{equation}
where $\xi_k >0$ is a steplength obtained with a line-search
 that satisfies the Armijo condition of sufficient decrease 
of $\hat J$'s value as follows 
\begin{equation}\label{armijo}
\hat J(\bar\alpha_k + \xi_k \, d_k) \le \hat J(\bar\alpha_k) +  \delta \,\xi_k \,
{( \nabla \hat J(\bar\alpha_k) , d_k)_{U}}, 
\end{equation}
where $0< \delta  < 1/2$; see \cite{nocedal}. Notice that we use the inner product 
of the $U=\R^n$ space.

%

For the formula of $\beta_k$ we use the formulation due to Dai and 
Yuan \cite{dai}
\begin{equation}
\beta_k^{DY}=\frac{( g_{k+1}, g_{k+1} )_U }{( d_k, y_k)_U} ,
\label{betadaiyuan}
\end{equation}
where $y_k=g_{k+1}-g_k$.

Summarizing, the NLCG scheme is implemented as follows 

\begin{algorithm}[NLCG Scheme] \label{alg:NLCG}
\noindent
\begin{itemize}
\item Input: initial approx. $\bar\alpha_0$, $d_0=-\nabla \hat J(\bar\alpha_0)$, 
index $k=0$, maximum $k_{max}$, tolerance $tol$. 
\begin{enumerate}
\item While ($k < k_{max}$ \&\& $\|g_k\|_{\R^\ell} > tol$  )  do 
\item Search the steplength $\xi_k >0 $, by sequentially shrinking, 
along $d_k$ satisfying (\ref{armijo});
\item Set $\bar\alpha_{k+1} = \bar\alpha_k + \xi_k\, d_k$. i.e. Eq. (\ref{ncgupdate}), 
according to the KKT condition, the 
eventually negative components of $\bar\alpha_{k+1}$ are set to $0$.
\item Compute $g_{k+1}=\nabla \hJ(\bar\alpha_{k+1})$ 
using Algorithm \ref{evaluate_gradient};
\item Compute $\beta_k^{DY}$ given by (\ref{betadaiyuan});
\item Let $d_{k+1}=-g_{k+1} + \beta_k^{DY} \, d_{k}$, i.e. Eq (\ref{ncgdir})
\item Set $k=k+1$;
\item End while
\end{enumerate}
\end{itemize}
\label{NCG}
\end{algorithm}

\paragraph{Correction factor for the logarithm in the objective.}

The numerical evaluation of the functional of Eq. (\ref{eq:functional}) 
has the problem of the logarithm in the points $X_l$ where 
the PDF at the final time has vanishing values.
Hence, the functional is replaced as follows
\begin{equation}\label{eq:functional2}
  J_\epsilon(f,\bar\alpha) = \frac{1}{L} \sum_{l=1}^L \log(\max( \epsilon, f(X_l,T,\bar\alpha)) ),
\end{equation}
with $\epsilon=10^{-12}$.

\paragraph{Nearest grid point for sample values }
The discrete PDF is defined on the mesh grid $\Omega_h$, the sample
values $X_l$ used for the evaluation of the PDF are approximated to the 
nearest values of the space mesh grid $\Omega_h$. This approximation affects both the
value of the functional and the terminal condition for the adjoint equation.

\paragraph{Von Mises distribution.}
The initial distribution $f_0(x)$ of Eq.(\ref{eq:fp}) is set as the following von Mises distribution
\begin{equation}
\rho(x;\mu,\kappa) = \frac{e^{\kappa\cos(2\pi (x-\mu-\Om_a)/(\Om_b-\Om_a)-\pi )}}{2\pi I_0(\kappa)},
\end{equation}
where $I_0(.)$ is the modified Bessel function of order $0$, and $\kappa$ is the concentration parameter that
should be taken large in order to approximate the Dirac delta function in zero as initial data for the forward PIDE.

\section{\label{sec:Cons} Numerical Tests} 

In this section we perform the non parametric estimation of L\'evy density distribution function,
that is to find the value $\bar\alpha = (\alpha_1, \ldots, \alpha_{\Nth})$ such that best fits
with the given data. We present two validation test cases and one application case to finance.

\paragraph{Testing for Consistency.}
We perform a numerical test on the consistency of our estimation procedure. 
According to the maximum likelihood technique, consistency here means that, if we fix a parametrization $N_\Theta$ and the parameter values, we can (approximately) reconstruct these values form our estimation procedure and maximization of the AIC, provides that a sufficiently large sample from the true distribution is given.  We simulate such a sample
using pseudo random realizations for the L\'evy process  $X(t)$. 
Details on the simulation methods can be found e.g. in \cite{Iac}. 

However note that in our case, cyclic boundary conditions have to be taken into account.
The data setting for our test case is as follows: the space domain $\Omega=[-\pi,\pi)$, the final time
$T=1$, the initial von Mises distribution has center $\mu = 0 $ and wideness $\kappa = 400$,
the drift of the stochastic process is $b=0$ and the Gaussian volatility is $\sigma = \sqrt{0.02}$.
The setting for the numerical solution is: space grid size $N=420$, time grid size $N_T=250$. 
The setting for the optimization is critical, we found the following parameters by the experience: initial approximation of the parameter rates $\bar\alpha_0=(0.1, 0.1, \ldots)$, constant of the Armijo condition $\delta=0.1$,
initial step-length of point 2. of Algorithm \ref{alg:NLCG} is set to $\xi_k=0.5$ and shrink 
by a factor $0.3$, $\xi_{k+1}=0.3\xi_k$.

As a first test, we perform a fit for a set of $L=10^5$ values generated by a Monte Carlo
algorithm for a simulated L\'evy process on the circle, with the following
five values of the jump rates: $\hat\alpha=\{3,2,1,0.5,0.25\}$.
We solve the fitting problem, i.e. calculating the estimates to $\alpha_1,\ldots,\alpha_{\Nth}$,
for different numbers of interpolatory functions: $\Nth=3,\ldots,7$.
The center  $\theta_1,\ldots,\theta_{\Nth} $ of the basis functions  $\Theta_j(x)$ are equally spaced 
in the domain $(-1,1)$ at the places $\theta_j=-1+j \De$, $j=1,\ldots, \Nth$, $\De=2/(\Nth+1)$,
this means the basis functions do not cover all the domain $\Om$.
In the following table the calculated value of $\{\alpha_j\}$ for each problem are reported 
versus $\Nth$
\begin{center}
\begin{tabular}{l|ccccc}
 & $\Nth=3$ & $\Nth=4$ & $\Nth=5$ & $\Nth=6$ & $\Nth=7$\\
 \hline\\
$\alpha_1$ & 3.4502 & 3.1771 & 2.9746 & 2.8580 & 2.8452\\
$\alpha_2$ & 1.1089 & 1.5577 & 1.8100 & 1.9922 & 2.1003\\
$\alpha_3$ & 0.4505 & 0.6576 & 1.0198 & 1.3025 & 1.5083\\
$\alpha_4$ &  & 0.3362 & 0.4951 & 0.7607 & 1.0077\\
$\alpha_5$ &  &  & 0.2490 & 0.3946 & 0.6137\\
$\alpha_6$ &  &  &  & 0.2042 & 0.3428\\
$\alpha_7$ &  &  &  &  & 0.1847
\end{tabular}
\end{center}

\vspace{3mm}

We see the good match for $\Nth=5$ with the original rates $\hat\alpha$.
In Figs. \ref{fig:data_fit3},\ref{fig:data_fit5} and \ref{fig:data_fit6} we can also 
appreciate the good data fitting between the calculated PDF and the histograms of
the simulated Monte Carlo data, for the proposed optimization problem with $\Nth=3,5,6$.

\begin{figure}
\centering
\includegraphics[width=0.45\textwidth, height=0.22\textheight]{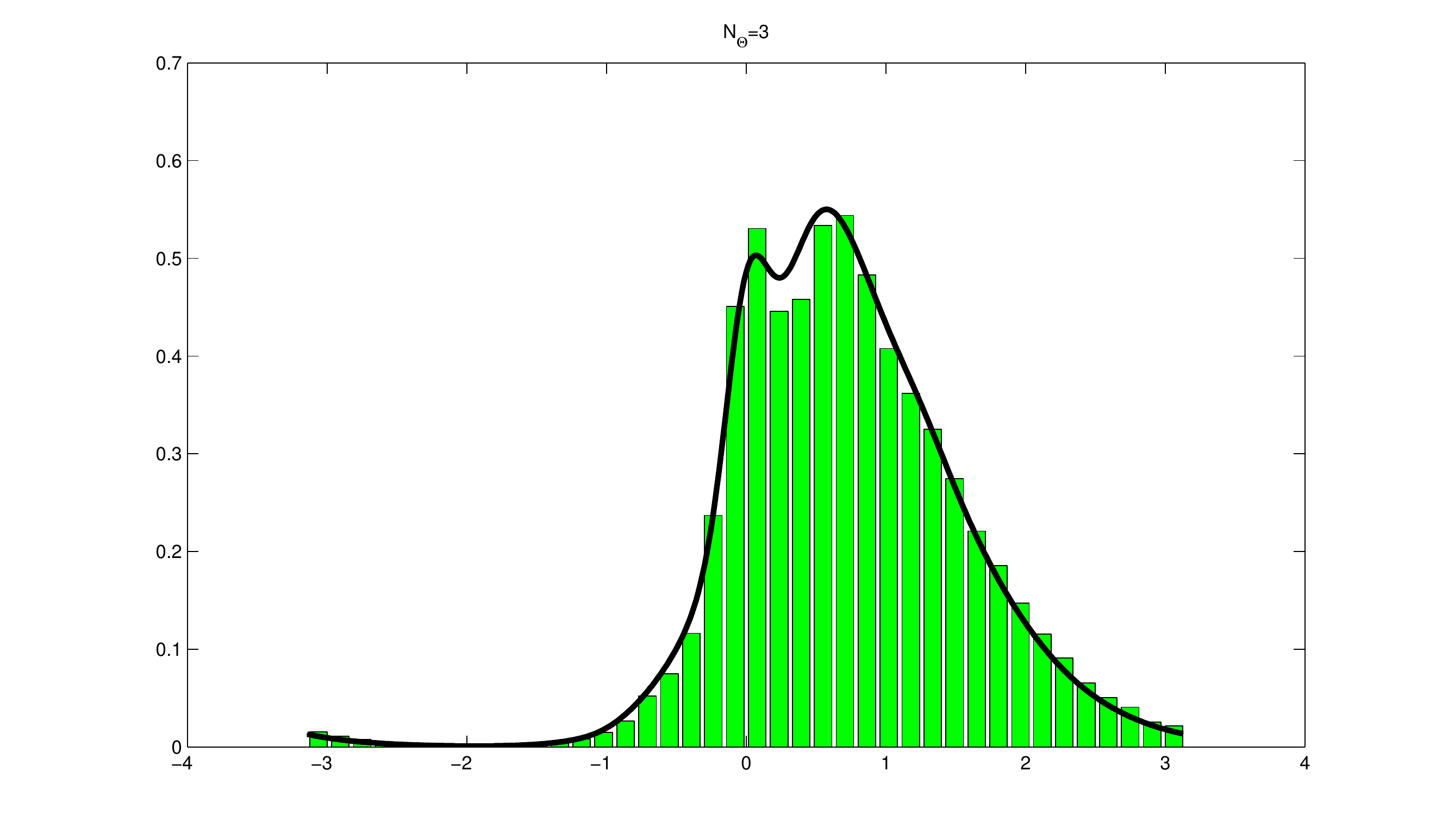}
\includegraphics[width=0.45\textwidth, height=0.22\textheight]{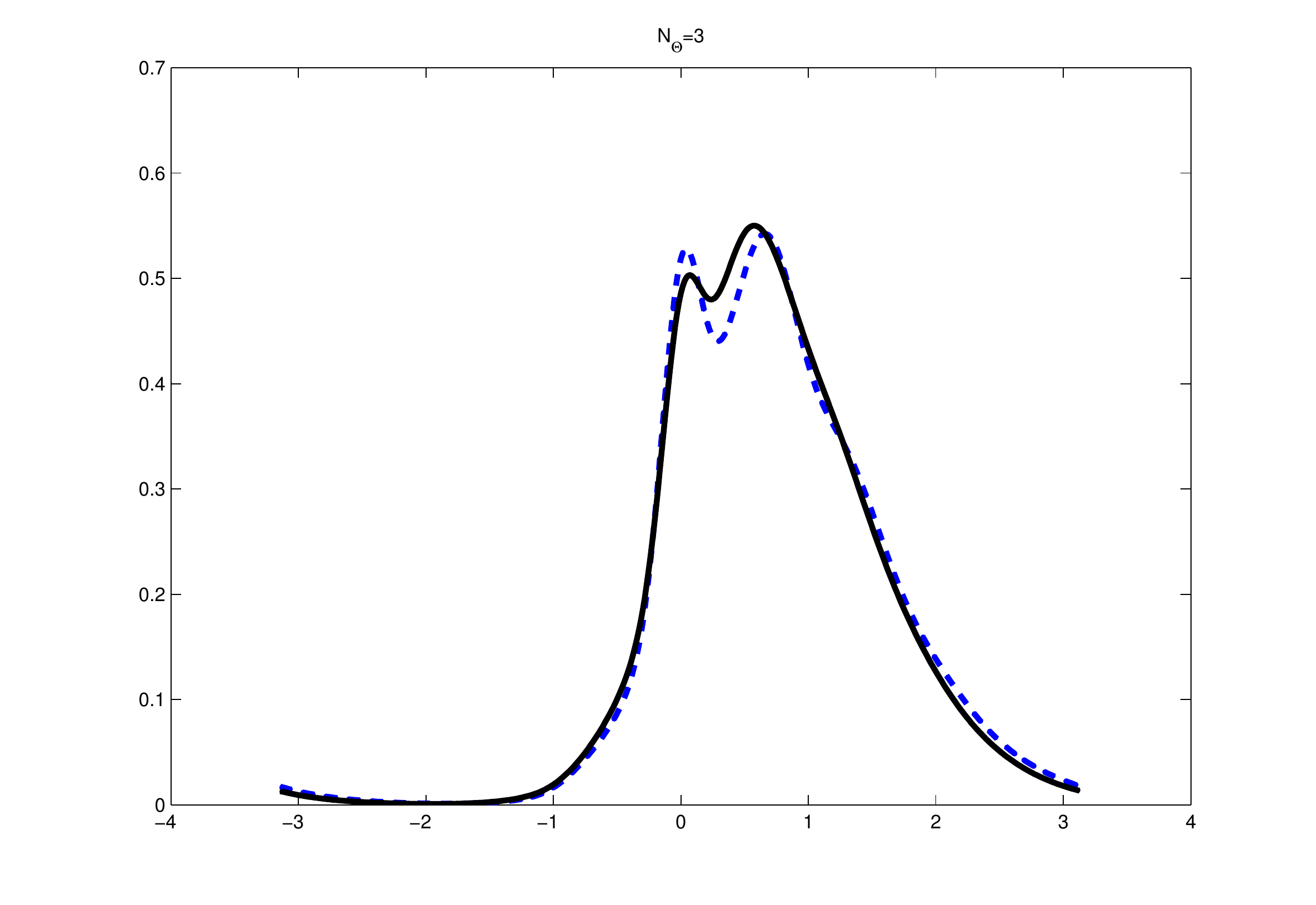}
\caption{Left. Result of the data fitting with $\Nth=3$ rates. 
Histograms: experimental L\'evy data collected in $40$ bins. 
Solid line calculated PDF. 
Right. Dashed line calculated PDF with the original $5$ rates.} \label{fig:data_fit3}
\includegraphics[width=0.45\textwidth, height=0.22\textheight]{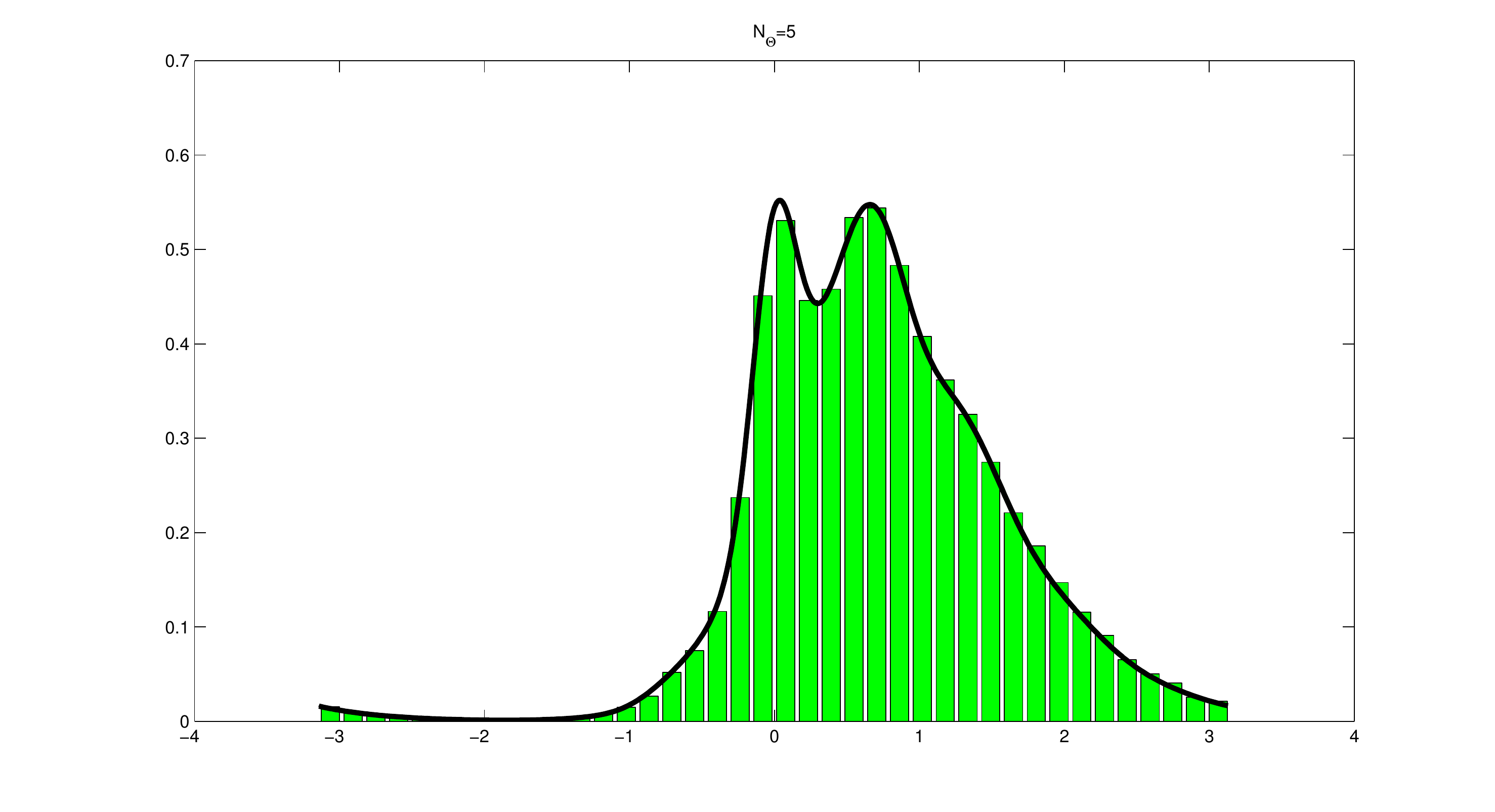}
\includegraphics[width=0.45\textwidth, height=0.22\textheight]{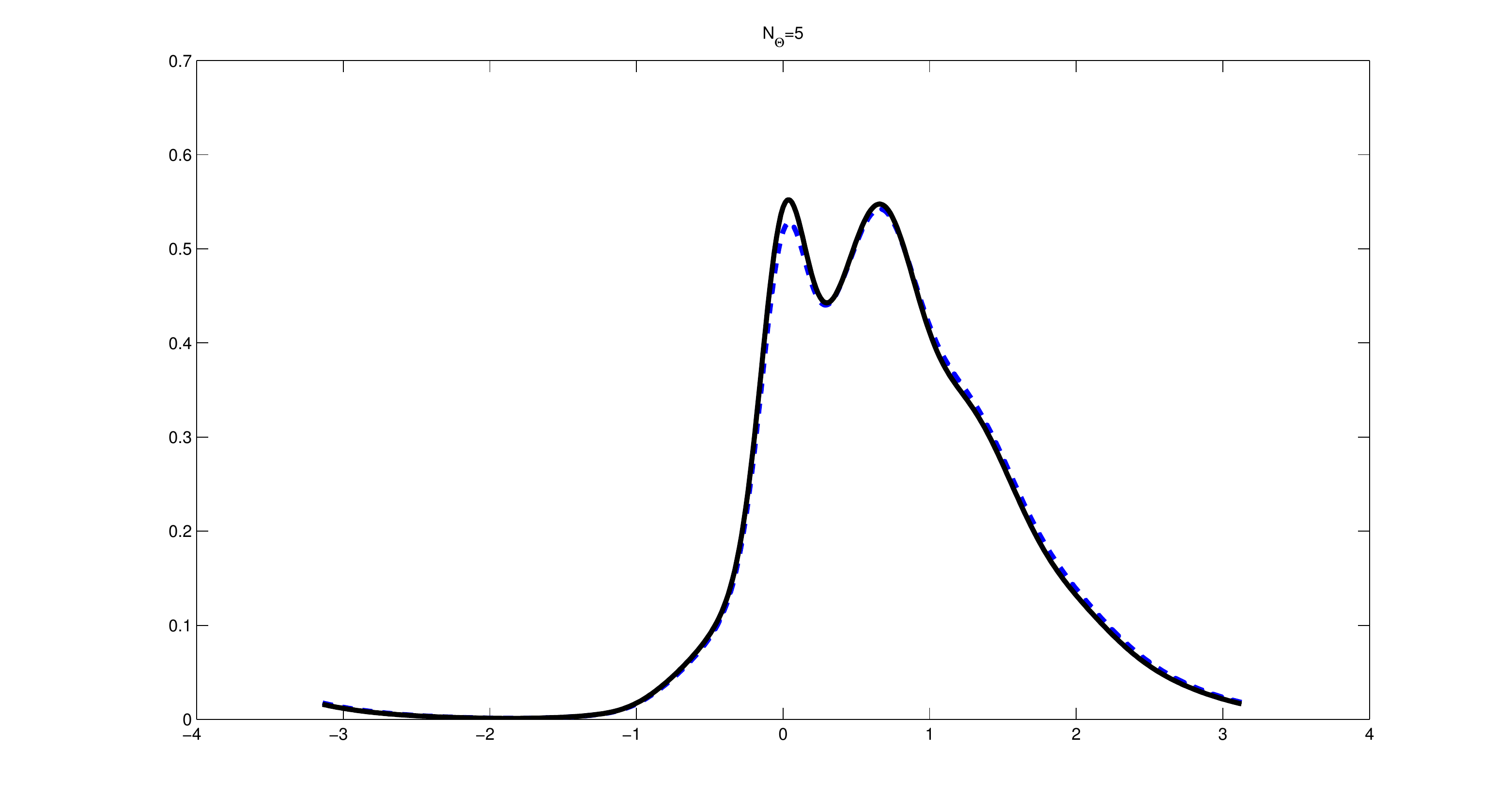}
\caption{Result of the data fitting with $\Nth=5$ rates. 
Histograms: experimental L\'evy data collected in $40$ bins. 
Solid line calculated PDF.
Right. Dashed line calculated PDF with the original $5$ rates.} \label{fig:data_fit5}
\includegraphics[width=0.45\textwidth, height=0.22\textheight]{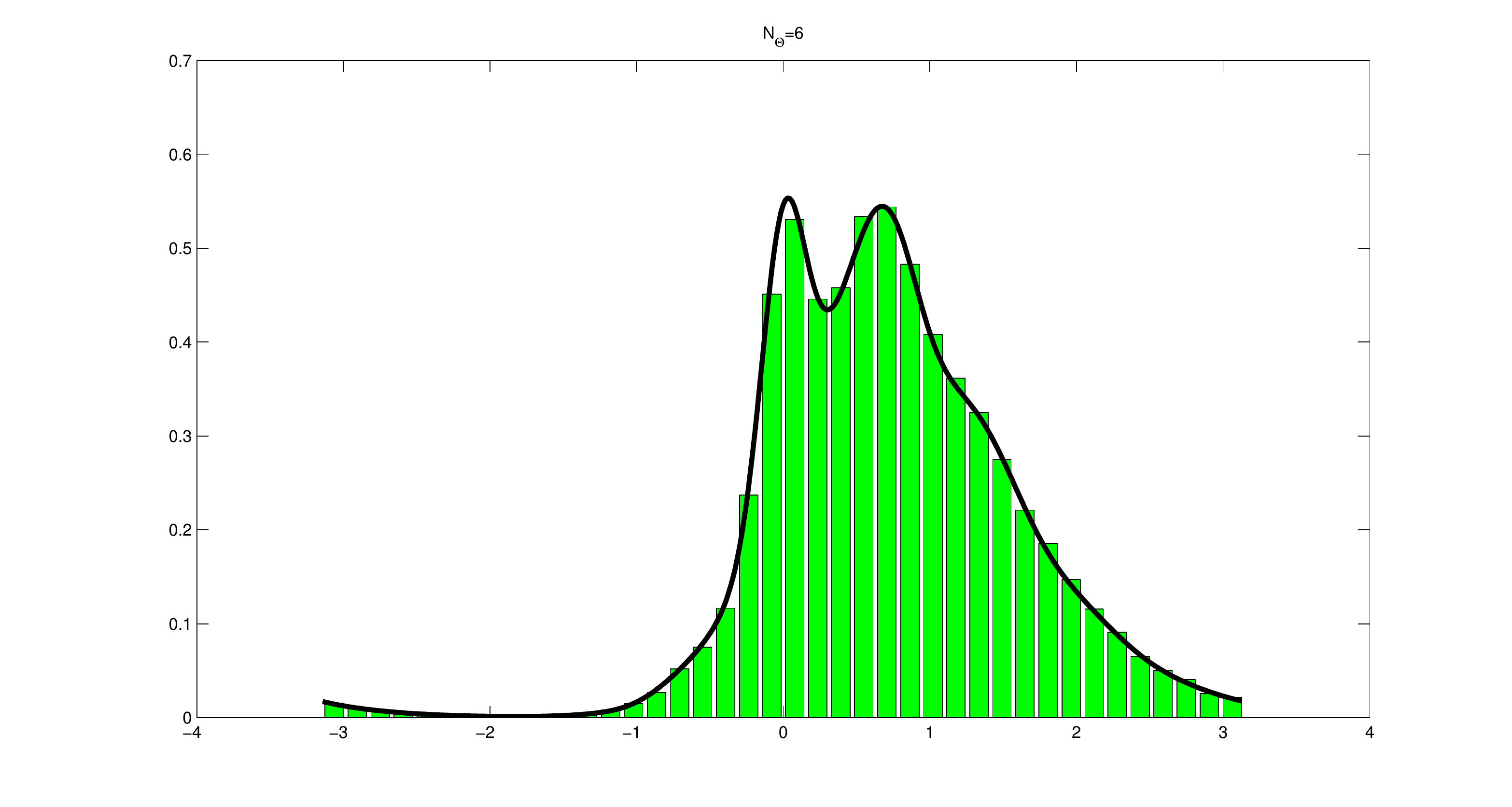}
\includegraphics[width=0.45\textwidth, height=0.22\textheight]{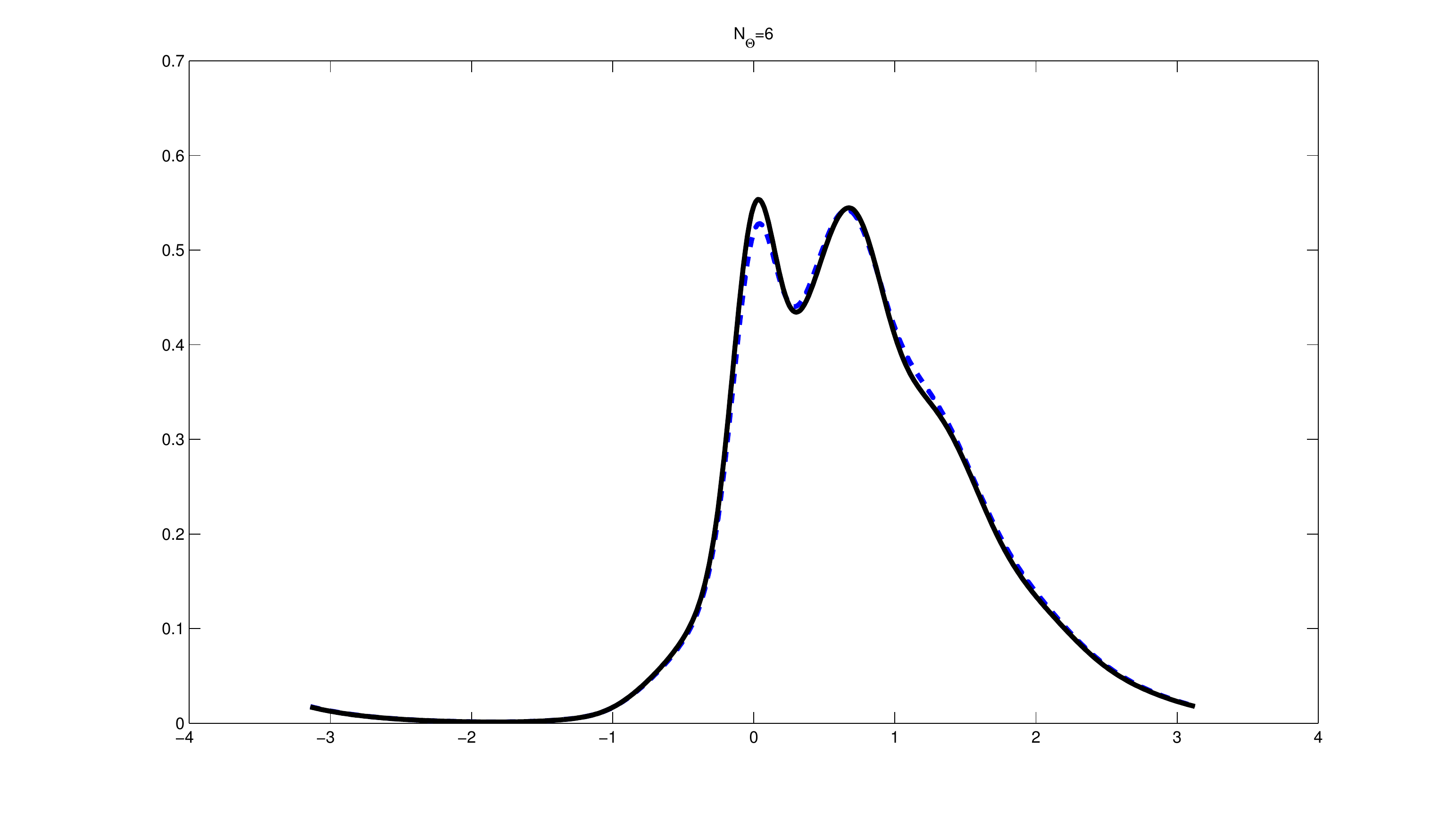}
\caption{Result of the data fitting with $\Nth=6$ rates. 
Histograms: experimental L\'evy data collected in $40$ bins. 
Solid line calculated PDF.
Right. Dashed line calculated PDF with the original $5$ rates.} \label{fig:data_fit6}
\end{figure}

Another interesting problem is the selection of the number of parameters $\Nth$ and the corresponding 
basis functions $\Theta_j$ for the best data fit. 
In Fig. \ref{fig:aic_test} we depict the result of the Akaike's Information Criterion (AIC) \cite{BA}, given by
\begin{equation}
AIC(\Nth) = L  J(f,\bar\alpha^*) - \log(\Nth).
\end{equation}
A common choice in statistics is to pick that parametrization that maximises the AIC. We can see that 
criterion gives the value $N_{\Th,opt}=6$, while the correct value is $5$. The difference in the AIC 
is however rather small for $\Nth$ between 5 and 7.

\begin{figure}
 \centering 
 \includegraphics[width=0.6\textwidth]{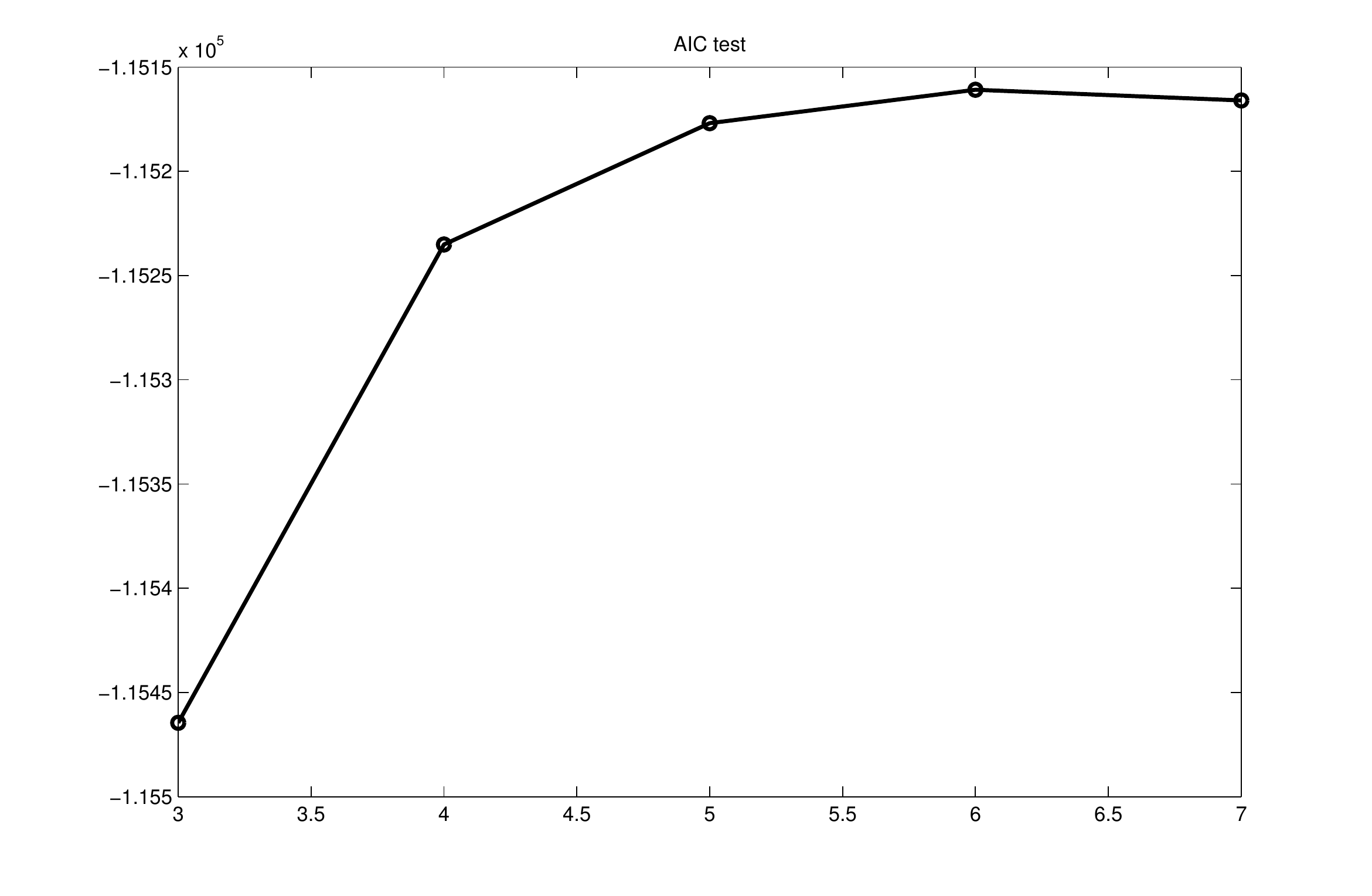}
\caption{Test for the appropriate regularisation with Akaike's Information Criterion}\label{fig:aic_test}
\end{figure}

\paragraph{Fitting Data from a bi-directional gamma process.}
In the second test we fit the final position at $T=1$ of $10^{5}$ samples of a stochastic process 
with the jumps distributed according a bi-directional gamma process with L\'evy measure $\nu$ on $\R$ 
given by the density \cite{App}
\begin{equation}
d\nu(s)=A\frac{e^{-\beta|s|}}{|s|} ds.
\end{equation}
Here $A>0$ is the so-called {\it shape} parameter and $\beta$ is the {\it rate} parameter.
Note that this is not a finite measure, so we are out of the compound Poisson class, and the trajectory of the 
bidirectional gamma process as infinitely many (small) jumps.
In \cite{App} only the unidirectional Gamma process is described. Let $Y^+(t)$ be such a 
unidirectional gamma process, then the L\'evy measure is
\begin{equation}
d\nu_+(s)=A\frac{e^{-\beta s}}{s}1_{\{s> 0\}}(s) ds \mbox{ and } 
f_{Y^+(t)}(y)=\frac{\beta^{At}}{\Gamma(At)}y^{At-1}e^ {-\beta y}1_{\{y> 0\}}(y).
\end{equation}
Let thus $Y^+(t)$ and $Y^-(t)$ be two independent copies of the Gamma process, then
\begin{equation}
Y(t)=Y^+(t)-Y^-(t)
\end{equation}
is our bi-directional gamma process, which is the jump part of our L\'evy process that also 
includes diffusion as in the first experiment.
If we project $Y(t)$ to the torus $[-\pi,\pi]$, the effect on the projected Levy measure $\phi_*\nu$,  see (\ref{eqa:wrappedLevyMeasure}), of the projected L\'evy process $X(t)=\phi(Y(t))$ is 
\begin{equation}
d\phi_*\nu(s)=\left(\sum_{n=0}^\infty A\frac{e^{-\beta(|s|+n\pi)}}{|s|+n\pi}\right)ds=\frac{A}{\pi}e^{-\beta|s|}\left(\sum_{n=0}^\infty \frac{e^{-\beta\pi n}}{\frac{|s|}{\pi}+n}\right)ds.
\end{equation}
Using
\begin{equation}
\sum_{n=0}^\infty \frac{e^{-qc}}{p+n}=e^{-q}\Phi(e^{-q},1,p),~~p\not=0, q>0,
\end{equation}
with $\Phi(z,s,a)$ the Lerch transcendent, for $p=|s|/\pi$ and $q=\beta\pi$, we get
\begin{equation}
d\phi_*\nu(s)=\frac{A}{\pi}e^{-\beta(|s|+\pi)}\Phi\left(e^{-\beta\pi},1,\frac{|s|}{\pi}\right)ds
\end{equation}
as the L\'evy measure on the torus for the projected bi-directional  Gamma process. 

In the simulations data are generated by scaling the rate parameter such that $1/\beta=1$,
and setting the shape to $A=0.5$. Finally, the data $Y(t)$ are projected to the torus $[-\pi,\pi)$.
In Fig. \ref{gamma_fit} we report the result of the AIC test the fit with $N_\Theta=9$
basis functions centered to $\theta_j=(-1+2(j-1)/9)\pi, j=1,\ldots 9$, whose the calculated rates are
$ \bar\alpha = (0, 0.0417, 0.0235, 0.1529, 0.8827, 0.8616, 0.1517, 0.0316, 0.0396)$. We conclude that our procedure results in high quality fits, even for L\'evy distributions that are not part of our hierarchy of parametrizations, but can only be approximated by these.  
\begin{figure}
\centering
 \includegraphics[width=0.45\textwidth, height=0.3\textheight]{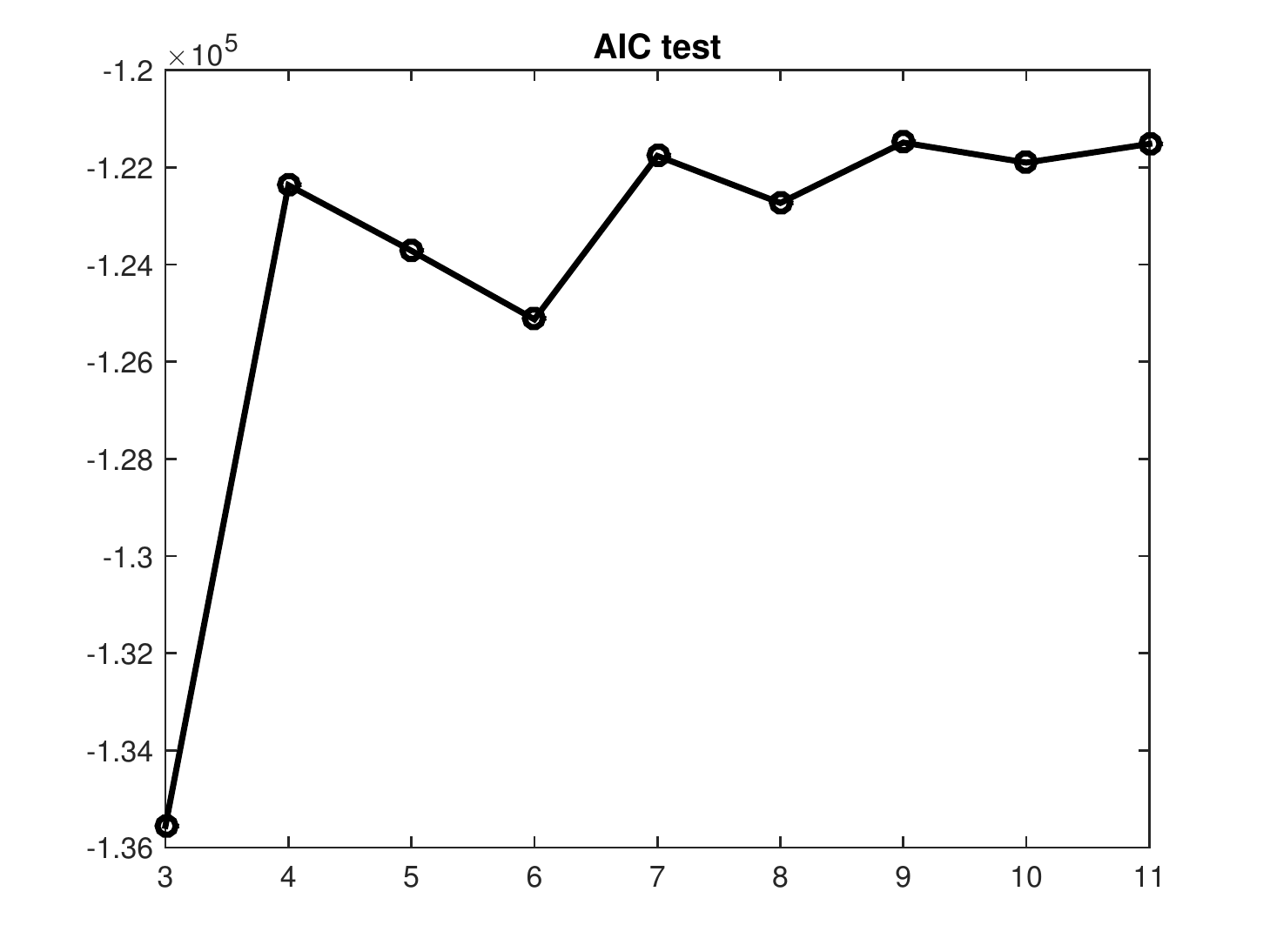}
\includegraphics[width=0.45\textwidth, height=0.3\textheight]{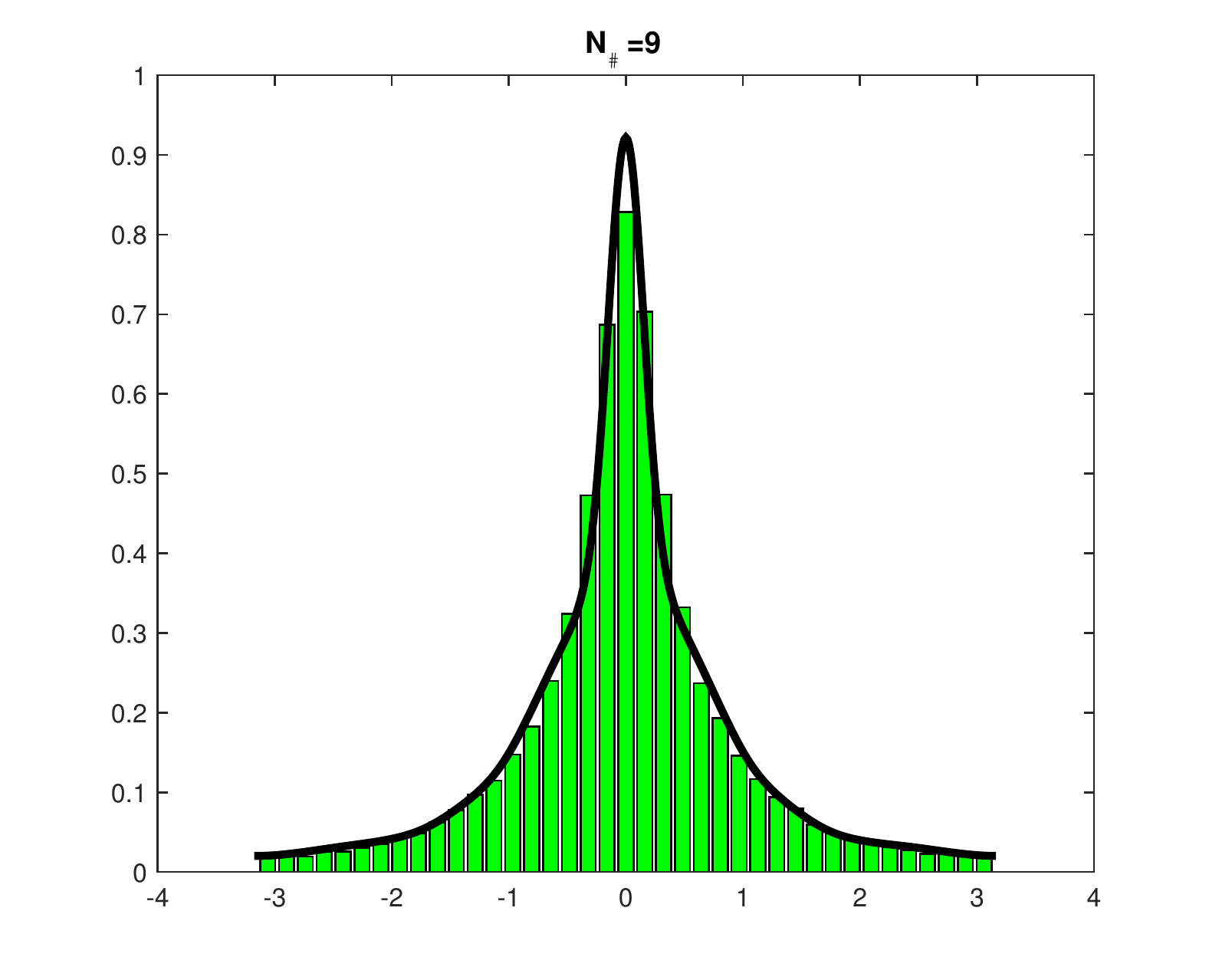}
\caption{Left, AIC test for the Gamma test. Right, fit with $N_\Theta=9$ basis functions.}
\label{gamma_fit}
\end{figure}

\paragraph{Financial Data.} As an example from a real world problem we report the result of fitting of the German Stock Exchange (DAX) index, which is publicly available, e.g.\ from the Yahoo Finance, see Figure \ref{dax_part} left panel. Within the data of all closing quotes between  April 1998 and March 2015, there are several periods of elevated volatility, so called volatility bursts. Obviously, this contradicts the description of the market with an exponential levy model $S(t)=e^{Y(t)}$ \cite{App,CT,Iac}, as the statistical law of $Y(t)-Y(s)$ does not only depend on $t-s$. In order to avoid the pitfalls of time dependent (or stochastic \cite{CT,Iac}) volatility, we identify a period of comparatively stable volatility of 1000 trading days between April 1998 and February 2002, see the right panel of Figure \ref{dax_part}. This data set has a small drift value $b= 6.787\times 10^{-4}$ which corresponds 
to the increased value of the stocks of 67\% nominal interest rate in 1000 trading days (followed by severe losses in the subsequent period). The empirical daily volatility (i.e. standard deviation of daily log-returns) in this period of time is $\sigma=9.304\times 10^{-3}$. 
The obvious absence of axial symmetry prohibits a Gaussian (Black-Scholes) market model from the outset.  Our goal is to find a suitable description of this sample with an exponential L\'evy market model from our hierarchy of parametrizations. 
\begin{figure}
 \begin{center}
 \includegraphics[width=0.47\textwidth]{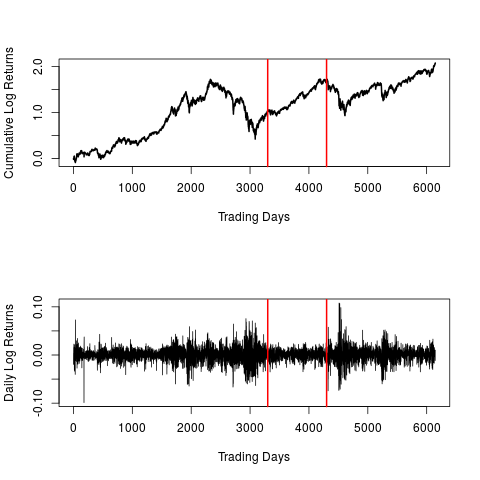}\includegraphics[width=0.47\textwidth]{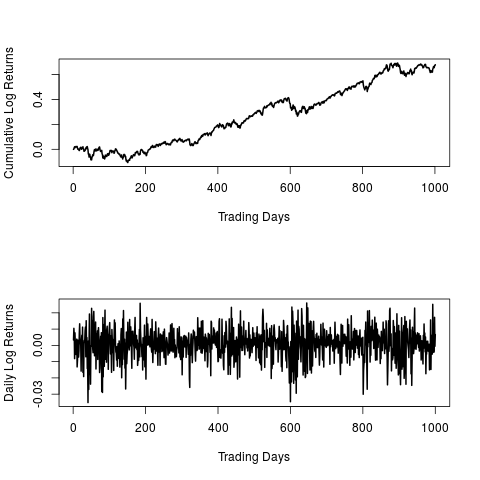}
\end{center}
\caption{Cumulative and daily log-returns of the DAX index between  26/11/1990 and  05/03/2015 (left panel) 16/4/1998 and 4/02/2002 (right panel). The vertical lines on the left panel correspond to the time period on the right panel, downloaded from finance.yahoo.com}\label{dax_part}
\end{figure}
The data has been mapped to the $[-\pi,\pi]$ torus, by rescaling and 'wrapping' the daily log-Returns
below/above 3\%, i.e. $[\Omega_a,\Omega_b]=[-0.03,0.03]$. Three data sets, all of them negative, were situated outside this band
\footnote{ Note that this corresponds to a loss of 10\% over four years being wrapped to the positive side. If the data is left skewed, as in the present sample, this might well introduce a bias in risk estimation to the optimistic side, if the procedure is used 'as is'. This can be mitigated with a larger torus such that no wrapping occurs, e.g. $[\Omega_a,\Omega_b]=[-0.05,0.05]$, see Section \ref{sec:Reg}.  
At the present stage, it is however not the intention of this work to provide a ready to use basis for risk estimates for financial applications.}.
The drift value is adapted to a re-scaled torus $[-\pi,\pi)$. Thus the drift on the torus of length $2\pi$ is $b=6.787\times 10^{-4}\pi/0.03= 0.07017$.

One fourth of the total empirical variance $8.655\cdot 10^{-5}$ is attributed for the 'fixed' diffusion, compare Section \ref{sec:Reg}, which yields 
a coefficient for the Laplace operator equals to $\pi^2 \mbox{(empirical Variance)}/(0.03)^2 /4/2=0.2372$ on the torus rescaled to $[-\pi,\pi)$.

With this setting we calculated the fitting of the distribution, with equally spaced basis functions
in the interval $[-\pi,\pi)$. In Fig. \ref{dax_fit} (left panel)  we report the result of the AIC test and the fit with $N_\Theta=6$ (right panel). The selected
basis functions are centred at $\theta_j=(-1+(j-1)/3)\pi, j=1,\ldots 6$, whose the calculated rates are
$ \bar\alpha = (0, 0, 0.484, 0.223, 0.304, 0)$. Although only three parameters are different from zero, the AIC is maximized at $\Nth=6$. That the AIC at $\Nth=3$ is lower is explained by the fact, that the more localized basis functions in the $\Nth=6$-basis are more adequate to fit the data. It is also a misinterpretation that the chosen parametrization misses an effective description with three parameters, since the position of the grid points are additional parameters.  Note that the zero entries of the 1st, 2nd and 6th slot $\bar \alpha$ actually correspond to small positive values and only represented as zero when rounded to the 3rd digit.

\begin{figure}
 \begin{center}
   \includegraphics[width=0.45\textwidth, height=0.3\textheight]{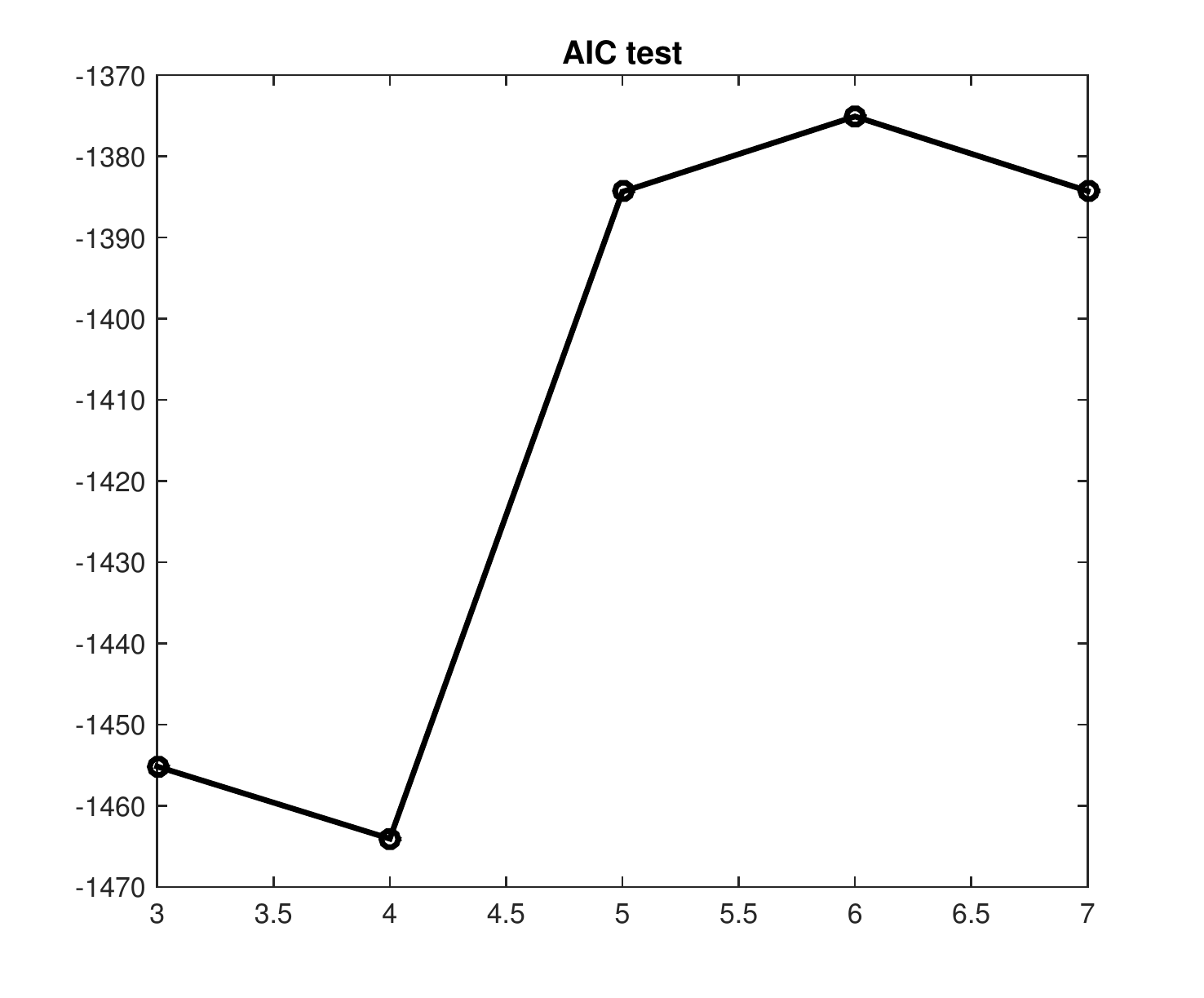}
 \includegraphics[width=0.45\textwidth, height=0.3\textheight]{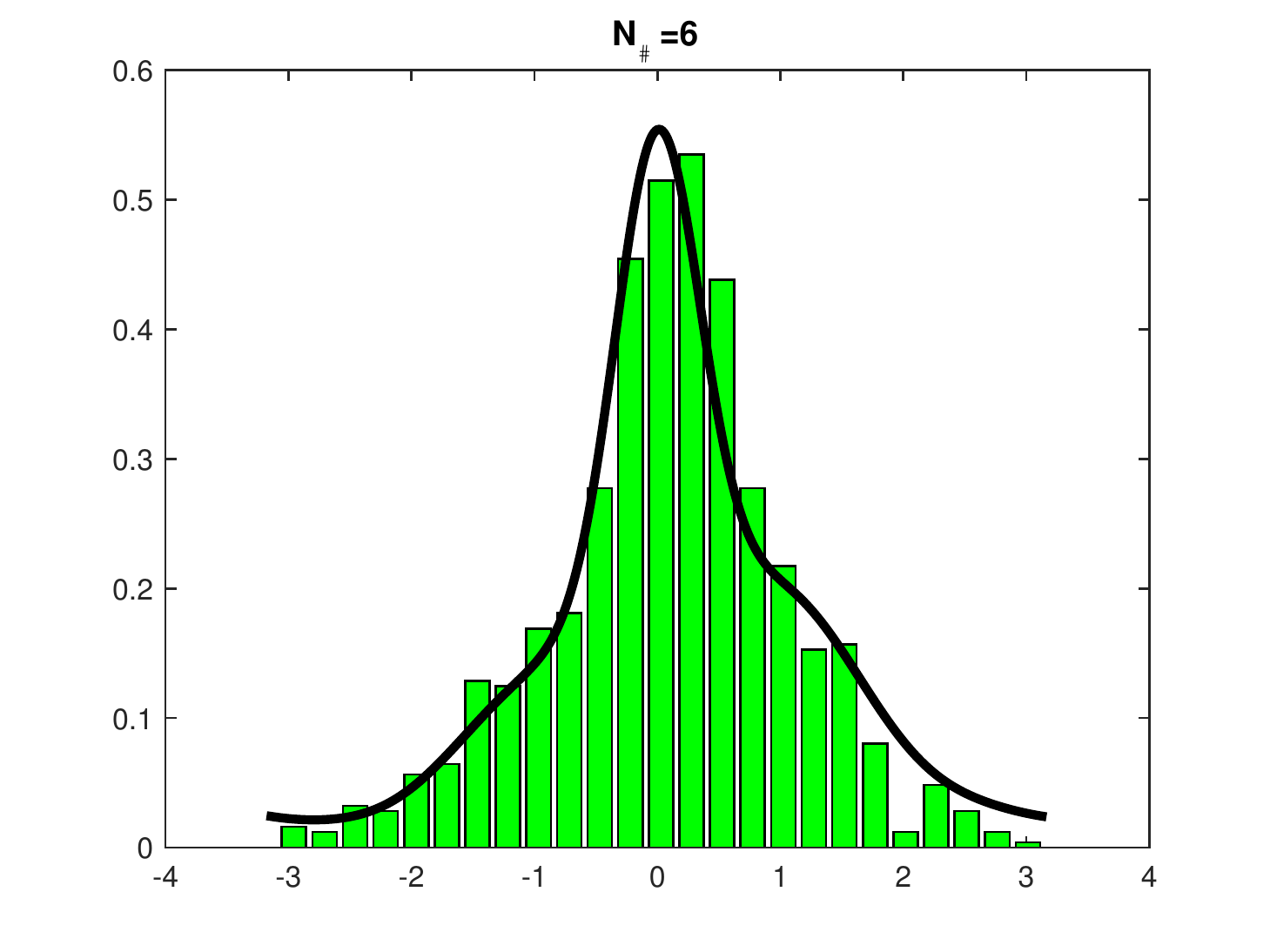}
 \end{center}
\caption{Left, AIC for the maximum likelihood estimate as a function of parameters for the DAX data (left). Right density for the Maximum-Likelihood fit with $N_\Theta=6$ basis functions, corresponding to the maximal AIC.}
\label{dax_fit}
\end{figure}

\section{\label{sec:Out}Conclusion and Outlook}
In the present article, we demonstrate the use of optimal control for PIDE for the non parametric estimation of L\'evy processes. Here the PIDE is given by Kolmogorov's forward equation which allows one to calculate the terminal distribution of the L\'evy process at time $T$. The objective functional is the log-likelihood evaluated on a sample of terminal values of the L\'evy process. 

Based on the study of L\'evy distributions, we set up approximate estimation problems 
that can be tackled by maximum likelihood estimation along with model selection based on Akaike's information criterion (AIC). As the density of L\'evy probability distributions in most 
cases can not be determined analytically, numerical solutions of the Kolmogorov forward equation (Fokker-Planck equation) and its backward (adjoint) analogue are needed for the efficient maximization of the log-likelihood functional. 

For the numerical solution of the optimality system we used the Chang-Cooper method with a
mid-point quadrature rule and second order backward time differentiation formula. This
numerical scheme is second order accurate and conservative, and we found conditions for 
stability and positivity of the numerical solution. We use a non linear
conjugate gradient method to find the optimality condition.

We have shown that this method works for spline discretizations of the density of the L\'evy measure 
with symmetric boundary conditions for up to 11 parameters. The results consistently fit simulated data from the family of discretizations itself. The same turns out to be true from L\'evy processes that only can be approximated by such discretizations, if the number of parameters goes to infinity, like the gamma process. Here the AIC provides an effective mechanism to choose an adequate discretization at a given sample size. Finally, we have demonstrated that also real-world, financial data can be effectively fitted using our strategy.

The future potential of this solution lies in the fact that, unlike FFT / spectral based calibration procedures 
that are widely used in financial engineering \cite{BR,CT,Iac}, the present approach naturally generalizes to 
processes that originate as the solution of Stochastic Differential Equations (SDE) with state dependent coefficients. Such local volatility models are frequently used in contemporary financial engineering.  

In this work, we used historic and low frequency data for non parametric model calibration. High frequency
historical data and implicit volatility data \cite{CT,Iac} are natural candidates to set up new objective functionals for 
related control problems that go beyond the control of the terminal distribution.  

Another relevant problem is the notorious occurrence of local minima in the maximum likelihood estimation. 
We expect this to be more severe, when the number of parameters significantly increases. An interesting hybrid 
approach would combine the robustness of non-parametric spectral calibration methods as a sort of pre-conditioner 
with the highly efficient maximum likelihood estimation.

\vspace{1cm}

\noindent {\bf Acknowledgements:} 
We would like to thank Alfio Borzi for interesting 
discussions and hospitality at the University of W\"urzburg.

\end{document}